\documentclass[a4paper]{article}			
\usepackage[latin1]{inputenc}				
\usepackage[T1]{fontenc}					
\usepackage{lmodern}						
\usepackage[english]{babel}					
\usepackage{amsmath,amsfonts}				
\usepackage{amssymb}						
\usepackage{stmaryrd}                       
\usepackage{multirow}						
\usepackage{graphicx}
\usepackage{epsf, float}
\usepackage{bm}								
\usepackage{amsthm}

\usepackage{a4wide} 						

\newtheorem{thm}{Theorem}
\newtheorem{lem}{Lemma}
\newtheorem{prop}{Proposition}
\newtheorem{conj}{Conjecture}
\newtheorem{coro}{Corollary}
\newtheorem{defi}{Definition}

\theoremstyle{definition}					
\newtheorem{rem}{Remark}
\newtheorem*{nota}{Notation}

\newcommand{\R}{\mathbb{R}}

\newcommand{\N}{\mathbb{N}}
\newcommand{\Z}{\mathbb{Z}}
\newcommand{\Q}{\mathbb{Q}}

\makeatletter
\def\imod#1{\allowbreak\mkern10mu({\operator@font mod}\,\,#1)}
\makeatother

\title{Rational Approximation and Arithmetic Progressions}
\author{Faustin ADICEAM\\
   Department of Mathematics, Logic House,\\ 
   National University of Ireland at Maynooth,\\
   Maynooth, Co. Kildare, Ireland.\\
   Email~:~\texttt{faustin.adiceam@maths.nuim.ie}}
\date{}

\begin{document}
\maketitle

\begin{abstract}
A reasonably complete theory of the approximation of an irrational by rational fractions whose numerators and denominators lie in prescribed arithmetic progressions is developed in this paper. Results are both, on the one hand, from a metrical and a non--metrical point of view and, on the other, from an asymptotic and also a uniform point of view. The principal novelty is a Khintchine type theorem for uniform approximation in this context. Some applications of this theory are also discussed.
\end{abstract}

\section{Introduction and statement of the main results}

Let $\xi$ denote an irrational number.

The celebrated Dirichlet's theorem in Diophantine approximation asserts that, for any real number $Q\ge 1$, there exist integers $p,q\in\Z$ such that 
\begin{align}\label{diruchlet} 
\left|\xi - \frac{p}{q}\right|\, \le\, \frac{1}{qQ} \quad \mbox{and} \quad 1\le q \le Q.
\end{align}
This \emph{uniform} version implies in particular an \emph{asymptotic} one, namely the fact that there exist arbitrarily large integer values of $q$ such that the inequality $|\xi- p/q|< q^{-2}$ holds true for some integer $p$ depending on $q$. Hurwitz has shown that the stronger inequality 
\begin{align*}
\left|\xi - \frac{p}{q}\right|\, \le\, \frac{1}{\sqrt{5}q^2}
\end{align*}
happens infinitely often and that the constant $1/\sqrt{5}$ in the right--hand side could not be chosen any smaller for the result to hold true for all irrationals.

In general, establishing a result concerning asymptotic approximation (and, \emph{a fortiori}, uniform approximation) when the numerators and/or the denominators of the rational approximants lie in given infinite sets turns out to be difficult (see e.g. Chapter 4 in~\cite{harmanmetricnbtheory} or~\cite{gorodkadyrov} and the references therein for some examples). This paper is concerned with the case where both the numerators and the denominators belong to prescribed arithmetic progressions. The known results in this context (which will be recalled), whether they are metrical, non--metrical, uniform or asymptotic, are very incomplete at the moment.

First some notation is fixed~: throughout, $a,b,r$ and $s$ will refer to integers satisfying the constraints
\begin{align}\label{contraintesdebase}
a\ge1, \quad b\ge1, \quad 0\le r \le a-1 \quad \mbox{and} \quad 0\le s \le b-1.
\end{align}
The problem under consideration amounts to finding rational approximations to an irrational $\xi$ with numerators (resp.~denominators) of the form $am+r$ (resp.~$bn+s$) for integers $n$ and $m$. Note that the case $r=s=0$ is settled in a straightforward manner~: applying Dirichlet's and Hurwitz's theorems to the irrational $b\xi/a$, it is easy to see that, on the one hand, for any integer $Q\ge b$, there exist integers $m$ and $n$ such that 
\begin{equation}\label{unifhomognonmetric}
\left|\xi-\frac{am}{bn} \right|\,\le\, \frac{ab}{(bn)Q} \quad \mbox{and} \quad 1\le bn\le Q
\end{equation} 
and that, on the other, there exist infinitely many integers $m$ and $n$ such that the inequality 
\begin{equation}\label{asymoptimhomognonmetric}
\left|\xi - \frac{am}{bn} \right|\, \le \, \frac{ab}{\sqrt{5}(bn)^2}
\end{equation} 
holds true infinitely often, the constant $(ab)/\sqrt{5}$ being optimal uniformly in $\xi\in\R\backslash\Q$. As will be apparent from the coming results, the fact that the constant $ab$ in the right--hand side of~(\ref{unifhomognonmetric}) may be chosen uniformly in $\xi\in\R\backslash\Q$ is typical of the ``homogeneous'' case $r=s=0$.

It is stressed that not all the theorems in this introduction are stated in full generality in order to keep the discourse coherent with respect to the problem under consideration.

\subsection{The theory of asymptotic approximation}

The first result deals with non--metrical asymptotic approximation. 

\begin{thm}\label{approxasympnonmetr}
Given an irrational $\xi$, there exist infinitely many integers $m$ and $n$ such that 
\begin{align}\label{inegapproxasympnonmetr}
\left|\xi - \frac{am+r}{bn+s} \right|\, \le \, \frac{ab}{4(bn+s)^2}
\end{align}
provided that $(r,s)\neq (0,0)$.
\end{thm}

This theorem has already been proved in some particular cases, for example with the additional constraint $a=b$ (cf.~\cite{uchiyama}) or with a  constant weaker than $(ab)/4$ on the right--hand side of~(\ref{inegapproxasympnonmetr}) (cf.~\cite{hartmanconditionsupp}). See also~\cite{Elsnermetricresult} and the references therein for further details and partial results in this direction.

\begin{rem}\label{rationelsapproxasympnonmetr}
Given the trivial relation $|u/v - p/q| \ge 1/(vq)$ satisfied by any two distinct rationals $u/v$ and $p/q$, an inequality as in~(\ref{inegapproxasympnonmetr}) can be satisfied by a rational $u/v$ infinitely often if, and only if, there exists $\alpha\in\Z$ (and hence infinitely many of those) such that $\alpha u \equiv r \imod{a}$ and $\alpha v \equiv s \imod{b}$, that is, from Lemma~\ref{lemelementaire} in subsection~\ref{soussectionresultatsauxiliaires} below, if, and only if, the three conditions $\gcd(bu, av)\,| \, (us-vr)$, $\gcd(u,a)\,|\, r$ and $\gcd(v,b)\,|\, s$ are simultaneously met.
\end{rem}

The next theorem deals with asymptotic approximation from a metrical point of view~: it provides a Khintchine type result in the setup under consideration. In what follows, $\lambda$ denotes the one--dimensional Lebesgue measure. As usual, a set is said to be of \emph{full measure} if the measure of its complement is null.

\begin{thm}\label{alakhint}
Let $\Psi~: [1, \infty) \rightarrow (0, \infty)$ be a non--increasing continuous function. Set $$\mathcal{K}\left(\Psi\right):= \left\{\xi\in\R \; : \; \left|\xi - \frac{am+r}{bn+s}\right|\,<\, \Psi(bn+s) \mbox{ i.o.}\right\},$$ where ``i.o.'' stands for ``infinitely often''.

Then,
\begin{align*}
\lambda\left(\mathcal{K}\left(\Psi\right)\right) = 
\begin{cases}
\texttt{ZERO} & \mbox{if } \sum_{n=1}^{\infty}n\Psi(bn+s) < \infty,\\\\
\texttt{FULL} & \mbox{if } \sum_{n=1}^{\infty}n\Psi(bn+s) = \infty.
\end{cases}
\end{align*}
Furthermore, the result still holds true if the additional condition $\gcd(am+r, bn+s) =\gcd(a,b,r,s)$ is also imposed in the definition of the set $\mathcal{K}\left(\Psi\right)$.
\end{thm}

In the case where congruential constraints are imposed only on the denominators of the appro\-xi\-mants (which corresponds to the case $a=1$ and $r=0$ in our setup), Theorem~\ref{alakhint} follows without much difficulty from the well--known theorem of Duffin and Schaeffer in Diophantine approximation as noticed by S.Hartman and Sz\"usz in~\cite{hartszusz}. On the other hand, in the case where both the nu\-me\-ra\-tors and the de\-no\-mi\-na\-tors belong to pre--assigned arithmetic progressions, the question was studied by G.Harman in~\cite{harmanrestrictedI} from the perspective of counting the number of solutions to Diophantine inequalities. Therefore, the main novelty in Theorem~\ref{alakhint} is the fact that the result holds true with the extra condition $\gcd(am+r, bn+s) =\gcd(a,b,r,s)$, which was a question left unanswered in~\cite{harmanrestrictedI}. It should be noted that the main feature of the proof of Theorem~\ref{alakhint} consists of establishing the optimal regularity of the set $\left\{(am+r)/(bn+s)\right\}_{n,m\in\Z}$ in $\R$. While this is a result interesting in its own right that can be used to simplify a great deal of G.Harman's proof, it does not follow in the same way as the optimal regularity of the rationals in $\R$ as soon as $r\neq 0$ or $s\neq 0$ (see subsection~\ref{asympmetric} for definitions and details).

An application of the Mass Transference Principle (see subsection~\ref{asympmetric}) allows one to translate Theorem~\ref{alakhint} into a result on the Hausdorff measure and dimension of the set $\mathcal{K}\left(\Psi\right)$. Here, $\mathcal{H}^t$ stands for the $t$--dimensional Hausdorff measure and $\dim$ for the Hausdorff dimension.

\begin{coro}\label{hausdmetric}
Let $t\in (0,1)$. Then, under the assumptions of Theorem~\ref{alakhint}, 
\begin{align*}
\mathcal{H}^t\left(\mathcal{K}\left(\Psi\right)\right) = 
\begin{cases}
0 & \mbox{if } \sum_{n=1}^{\infty}n\Psi(bn+s)^t < \infty,\\\\
\infty  & \mbox{if } \sum_{n=1}^{\infty}n\Psi(bn+s)^t = \infty.
\end{cases}
\end{align*}
In particular, $\dim \left(\mathcal{K}\left(\Psi\right)\right) = \inf \left\{t>0\; : \; \sum_{n=1}^{\infty}n\Psi(bn+s)^t < \infty \right\}$.

This result still holds true with the additional condition $\gcd(am+r, bn+s) =\gcd(a,b,r,s)$ in the definition of the set $\mathcal{K}\left(\Psi\right)$.
\end{coro}

\subsection{The theory of uniform approximation}

Even though the introduction of the concept of \emph{hat exponent} (see e.g.~\cite{buglaursturm}) has made the distinction between uniform problems and asymptotic problems more systematic in Diophantine approximation, results on uniform approximation under constraints remain quite rare in the literature~: one can for instance mention the recent work of Chan in~\cite{chan} on uniform approximation by sums of two rationals or the work of Dodson, Rynne and Vickers showing in~\cite{dodrynvic} that if $\mathcal{M}\subset \R^k$ ($k\ge 3$) belongs to a general class of smooth manifolds then, for almost all points lying on $\mathcal{M}$ (with respect to the induced measure), Dirichlet's Theorem cannot be infinitely improved in some sense made precise in the paper. 

However, in the case where the numerators and the denominators of the approximants are subject to congruential constraints as those under consideration so far, a reasonably complete theory of uniform approximation can be established, both from a metrical and a non--metrical point of view. This is the subject of this subsection. To this end, a few definitions are first introduced.

\begin{defi}\label{defiunifapproxarith}
Given a function $\Psi~: [1, \infty) \rightarrow (0, \infty)$, a real number $\xi$ is said to admit a \emph{$\Psi$--uniform $(a,b,r,s)$--approximation} if there exists $Q_0\ge 1$ such that, for any integer $Q \ge Q_0$, there are integers $m$ and $n$ such that $$\left|\xi - \frac{am+r}{bn+s} \right|\, \le \, \frac{\Psi(Q)}{bn+s} \quad \mbox{and } \quad 1\le bn+s \le Q.$$ The set of real numbers admitting a $\Psi$--uniform $(a,b,r,s)$--approximation will be denoted by $\mathcal{U}(\Psi)$.

Furthermore, $\xi\in\R$ will be said to admit a uniform $(a,b,r,s)$--approximation with exponent $\mu\in [0,1]$ if there exists $c>0$ such that $\xi \in \mathcal{U}\left(Q\mapsto cQ^{-\mu}\right)$.
\end{defi}

From a non--metrical point of view, a necessary and sufficient condition, explicit in terms of the continued fraction expansion, can be given for an irrational $\xi$ to be uniformly approximable at order $\Psi$ up to an explicit constant depending on $\xi$ (that is, for there to exist $c:=c(\xi)$ such that $\xi\in\mathcal{U}(c\Psi)$). In what follows, the sequence of the partial quotients of $\xi$ (resp.~of its convergents) will be denoted by $(a_k(\xi))_{k\ge 0}$ or by $(a_k)_{k\ge 0}$ for the sake of simplicity (resp.~by $(p_k(\xi)/q_k(\xi))_{k\ge 0}$ or by $(p_k/q_k)_{k\ge 0}$), with $a_0\in\Z$ and $a_k\in\N$ for $k\ge 1$ (here and throughout, $\N$ will refer to the set of \emph{positive} integers). The necessary and sufficient condition is technical by nature and is concerned with the indices $k\ge 1$ for which the relations  
\begin{align}\label{conditionfondaapproxunif}
\gcd(p_{k-1},a)\,|\, r, \quad \gcd(q_{k-1},b)\,|\, s \quad \mbox{ and } \quad \gcd(bp_{k-1},aq_{k-1})\,|\, (sp_{k-1}-rq_{k-1}),
\end{align}
are not simultaneously satisfied.

\begin{thm}\label{CNSapproxunifarithnonmetr}
\sloppy Let $\xi$ be an irrational number given by its continued fraction expansion $\xi =\left[a_0;a_1, \dots \right]$. Let $\Psi~: [1, \infty) \rightarrow (0, \infty)$ be a continuous non--increasing function. Set 
\begin{equation}\label{psitilde}
\widetilde{\Psi}~: Q\in [1, \infty) \mapsto Q\Psi(Q)\in(0, \infty)
\end{equation} 
and assume that there exist $\gamma>0$, $\kappa \ge 1$ and $\eta\ge 1$ satisfying
\begin{align}\label{conditiondecroissancepsi}
\inf_{Q\ge 1}\widetilde{\Psi}(Q)\, \ge \,\gamma, \quad \widetilde{\Psi}(Q)\, \le \, \kappa \widetilde{\Psi}(2Q) \quad \mbox{ and } \quad \widetilde{\Psi}(Q)\, \le \, \eta \widetilde{\Psi}(ab(Q+1))\quad \mbox{ for all } Q\,\ge \, 1.
\end{align}
Then there exists a constant $c:=c(\xi)>0$ such that $\xi\in\mathcal{U}(c\Psi)$ if, and only if, there exists an integer $M\ge1$ such that for all indices $k\ge 1$ for which conditions~(\ref{conditionfondaapproxunif}) are \textbf{\emph{not}} met, one has $a_k \,\le \, M \widetilde{\Psi}(q_k)$.

Furthermore,  
\begin{align}\label{cxiexpression}
c(\xi)=8(ab)^2\kappa\eta \max\{4M, \gamma^{-1} \}\quad \mbox{ and } \quad Q_0=ab
\end{align}
are admissible values, where $Q_0$ is the parameter introduced in Definition~\ref{defiunifapproxarith}.
\end{thm}

\begin{rem}$\qquad$
\begin{itemize}
\item The existence of $\kappa$ together with the assumption of the monotonicity of $\Psi$ actually implies the existence of $\eta$ in~(\ref{conditiondecroissancepsi}). However, the explicit presence of these two constants makes the definition of $c(\xi)$ in~(\ref{cxiexpression}) more effective.

\item Conditions~(\ref{conditiondecroissancepsi}) should be seen as an attempt to remove any assumption of monotonicity on the function $\widetilde{\Psi}$~: indeed, it is easily checked that they are automatically satisfied if $\widetilde{\Psi}$ is assumed to be non--decreasing (with $\gamma = \widetilde{\Psi}(1)$ and $\kappa= \eta = 1$). 

\item The existence of the constants $\kappa$ and $\eta$ (which is ensured for a fairly large class of functions --- for instance any function rational in $\log Q$ and $Q$) means that the function $\widetilde{\Psi}$ does not admit abrupt variations. It is a weaker assumption that the usual one when trying to remove the assumption of monotonicity~: transposed in this context, the latter would ask that, for every  $c>1$, $\widetilde{\Psi}(cQ)<c\widetilde{\Psi}(Q)$ for all $Q\ge 1$ (see for instance~\cite{casselsapproxdioph} and \S 4.1 in~\cite{crelledettasanju}).

\item The existence of the constant $\gamma$ is a relatively mild restriction. Indeed, it is well--known that if a real number $\alpha$ satisfies Dirichlet's theorem with $(2Q)^{-1}$ as the approximating function instead of $Q^{-1}$ (that is, if the right--hand side of the first inequality in~(\ref{diruchlet}) is replaced by $(2qQ)^{-1}$), then $\alpha$ has to be rational (see for instance Lemma~6 in~\cite{waldrencentdiophapprox} for a proof). This implies in particular that the function $\widetilde{\Psi}$ in Theorem~\ref{CNSapproxunifarithnonmetr} cannot tend to zero.

\item Condition~(\ref{conditionfondaapproxunif}) obviously holds true in the ``homogeneous case'' $r=s=0$, in which case one finds again the aforementioned result on uniform approximation with exponent 1 where the constant $c(\xi)=ab$ was proved to be admissible for all $\xi\in\R\backslash\Q$.

\item Any badly approximable number has uniformly bounded partial quotients regardless of whether condition~(\ref{conditionfondaapproxunif}) is met or not. Therefore, all badly approximable numbers admit a uniform $(a,b,r,s)$--approximation with exponent 1. This shows in particular that the set of real numbers for which Dirichlet's theorem holds true up to a constant in the context of $(a,b,r,s)$--approximation has full Hausdorff dimension. In the case of badly approximable numbers, the existence of a uniform $(a,b,r,s)$--approximation with exponent 1 will be proved to be a direct consequence of the three distance theorem in subsection~\ref{nonmetrunif}.

\item It will be clear that the proof of Theorem~\ref{CNSapproxunifarithnonmetr} can be adapted to show that, given any $\mu\in (0,1]$, there always exists an irrational $\xi$ such that $\xi$ does not admit a uniform $(a,b,r,s)$--approximation with exponent $\mu$ as soon as $r\neq 0$ or $s\neq 0$.
\end{itemize}
\end{rem}

From a metrical point of view, the only known result in the context of uniform $(a,b,r,s)$--approximation seems to be that of S.Hartman who proved in~\cite{hartfeaturediri} that almost no real number satisfies Dirichlet's theorem if the denominators of the approximants were prescribed to be odd. The following corollary of Theorem~\ref{CNSapproxunifarithnonmetr}, which is very much the main result of the paper, provides a reasonably complete answer to this problem. It constitutes the first example of a Khintchine type result in the context of uniform approximation. The reader should note the differences with respect to a standard Khintchine type result as Theorem~\ref{alakhint}.

\begin{coro}\label{resultatprincikhintchiunif}
Let $\Psi~: [1, \infty) \rightarrow (0, \infty)$ be a continuous non--increasing function such that the function $\widetilde{\Psi}$ as defined by~(\ref{psitilde}) is non--decreasing.

If $r\neq 0$ or $s\neq 0$, then
\begin{align*}
\lambda\left(\mathcal{U}\left(\Psi\right)\right) = 
\begin{cases}
\texttt{ZERO} & \mbox{if } \sum_{Q=1}^{\infty}\frac{1}{Q^2\Psi(Q)} = \infty\\\\
\texttt{FULL} & \mbox{if } \sum_{Q=1}^{\infty}\frac{1}{Q^2\Psi(Q)} < \infty.
\end{cases}
\end{align*}
\end{coro}

\sloppy Thus, as soon as $r\neq 0$ or $s\neq 0$, almost no real number admits a uniform $(a,b,r,s)$--approximation with exponent 1. This also holds true if one takes $\Psi(Q) = \log Q/ Q$ as the approximating function. On the other hand, almost all real numbers belong to the set $\bigcap_{n=1}^{\infty} \mathcal{U}\left(Q\mapsto (\log Q)^{1+1/n}Q^{-1} \right)$.

\paragraph{}
The paper is organized as follows~: the results on asymptotic approximation (Theorems~\ref{approxasympnonmetr} and~\ref{alakhint} and Corollary~\ref{hausdmetric}) are proved first in section~\ref{asympapproxpreuve}. Then proofs for Theorem~\ref{CNSapproxunifarithnonmetr} and Corollary~\ref{resultatprincikhintchiunif}, dealing with uniform approximation, will be provided in section~\ref{unifapproxpreuve}. Finally, various applications of Diophantine approximation with congruential constraints on both the numerator and the denominator of the approximants will be mentioned in section~\ref{applicationapproxarithm}. In particular, applications to the estimate of some trigonometrical functions and to so--called visibility problems in geometry will be considered.

\section{Proofs of the results related to asymptotic approximation}\label{asympapproxpreuve}

Theorem~\ref{approxasympnonmetr}, Theorem~~\ref{alakhint} and Corollary~\ref{hausdmetric} are proved in this section.

\subsection{Non--metrical point of view}

We first begin with a proof of Theorem~\ref{approxasympnonmetr}. This can actually be seen as a consequence of Minkowski's theorem on the product of two linear forms (see for instance Theorem~1 p.46 in~\cite{casselsapproxdioph}).

\begin{proof}[Proof of Theorem~\ref{approxasympnonmetr}]
Let $\xi\in\R\backslash\Q$. Consider the linear forms $L_1(x,y) = by$ and $L_2(x,y)= b\xi y - a x$ with determinant $\Delta = -ab$ and set $\eta:=s$ and $\nu:=s\xi-r$.  From Minkowski's theorem on the product of two linear forms, there exist integers $m$ and $n$ such that $$\left|L_1(n)+\eta\right|. \left|L_2(m,n)+\nu \right| \, = \, \left|bn+s\right| .\left|\xi(bn+s)-(am+r) \right| \, \le \, \frac{ab}{4}\cdotp$$ As $\xi$ is irrational, given $\epsilon >0$, one can furthermore add the constraint that $$\left|L_2(m,n)+\nu \right| \, := \, \left|(b\xi n-am)+(s\xi-r) \right| \, < \, \epsilon$$ (see for instance Theorem~1 p.46 from~\cite{casselsapproxdioph} for details). Since $\nu:=s\xi -r \not\in b\xi\Z+a\Z$, one gets infinitely many pairs of integers $(m,n)\in\Z^2$ satisfying~(\ref{inegapproxasympnonmetr}) by letting $\epsilon$ tend to zero.
\end{proof}

\begin{rem}\label{generalizthapproxasymnonmetr}
Theorem~\ref{approxasympnonmetr} can be generalized to the case of inhomogeneous approximation in the following way~: for any $\xi\in\R\backslash\Q$ and any $\alpha\in\R$, there exist infinitely many pairs $(m,n)\in\Z^2$ such that the inequality $$\left|\xi(bn+s)-(am+r)+\alpha\right|\, \le \, \frac{ab}{4\left|bn+s\right|}$$ holds if $s\xi+r+\alpha\not\in b\xi\Z + a\Z$ (this follows readily from the previous proof). If, however, $s\xi+r+\alpha\in b\xi\Z + a\Z$, the situation is essentially the same as the ``homogeneous'' case $r=s=0$ and it is easily seen, using for instance~(\ref{asymoptimhomognonmetric}), that the result still holds true upon choosing some constant bigger than $ab/4$ depending on $\alpha$ in the right--hand side of the inequality.
\end{rem}

\paragraph{} 
A natural question related to Theorem~\ref{approxasympnonmetr} is whether the constant $(ab)/4$ appearing on the right--hand side of~(\ref{inegapproxasympnonmetr}) is optimal. This has been proved by Eggan in Theorem 3.2 from~\cite{egganapproxdioph} (following ideas due to Cassels --- see the proof of Theorem II B p.49 in~\cite{casselsapproxdioph}) in the case when the parity of the numerators and the denominators of the rational approximants were prescribed in a non--trivial way (that is, when $a=b=2$ and $r\neq 0$ or $s\neq 0$). It is therefore tempting to set the following conjecture, on which the author plans to come back in the near future.

\begin{conj}\label{conjabsur4}
If $r\neq 0$ or $s\neq 0$, the constant $(ab)/4$ appearing on the right--hand side of~(\ref{inegapproxasympnonmetr}) cannot be improved uniformly in $\xi\in\R\backslash\Q$.
\end{conj}

\subsection{Metrical point of view}\label{asympmetric}

A proof is now provided for Theorem~\ref{alakhint} . The notation from this theorem is kept in this subsection. Since the set $\mathcal{K}\left(\Psi\right)$ is clearly invariant by translation by a multiple of the integer $a$, it suffices to establish the Khintchine type result for the set $\mathcal{K}\left(\Psi\right)\cap (0,a)$ which, for the sake of simplicity, shall still be denoted by $\mathcal{K}\left(\Psi\right)$ in what follows.

The convergence part of Theorem~\ref{alakhint} can be obtained in a classical way as a consequence of the Borel--Cantelli lemma~: details are left to the reader (see for instance p.13 of~\cite{bugeaudlivre1}). In order to prove the divergence part, the concept of an \emph{optimal regular system} is introduced. Recall that $\lambda$ denotes the one--dimensional Lebesgue measure.

\begin{defi}\label{defioptiregulsyst}
Let $E\subset\R$ be a bounded open interval and let $\mathcal{S}:=(\alpha_j)_{j\ge 1}$ denote a sequence of distinct real numbers. 

The sequence $\mathcal{S}$ is an \emph{optimal regular system of points in $E$} if there exist positive constants $c_1$ and $c_2$ depending only on $\mathcal{S}$ and, for any interval $I$ contained in $E$, a number $K_0$ depending on $\mathcal{S}$ and $I$ such that the following property holds true~: for any $K\ge K_0$, there exist integers $1\le i_1< \dots < i_t \le K$ with $\alpha_{i_h}\in I$ for $h=1, \dots , t$ satisfying $$\left|\alpha_{i_h}-\alpha_{i_l}\right|\, \ge \, \frac{c_1}{K} \;\mbox{ for }\; 1\le h\neq l \le t \quad \mbox{and} \quad  t \, \ge \, c_2\lambda(I)K.$$
\end{defi}

The next theorem, due to Beresnevich in~\cite{beres1onapprox} and~\cite{beres2concept} (see also Chapter 6 of~\cite{bugeaudlivre1}), shows that the set of real numbers close to infinitely many points in an optimal regular system satisfies the divergent part of a Khintchine type statement.

\begin{thm}[Beresnevich]\label{thmberesnoptiregusyst}
Let $E$ be a bounded interval and let $\mathcal{S}:=(\alpha_j)_{j\ge 1}$ denote an optimal regular system in $E$. Given a non--increasing continuous function $\Psi~: [1, \infty) \rightarrow (0, \infty)$, define the set $\mathcal{K}_{\mathcal{S}}\left(\Psi\right)$ as $$\mathcal{K}_{\mathcal{S}}\left(\Psi\right):= \underset{j\rightarrow \infty}{\limsup}\left\{\xi\in E \; : \; \left|\xi- \alpha_j\right|\, < \, \Psi(j)\right\}.$$
Then the set $\mathcal{K}_{\mathcal{S}}\left(\Psi\right)$ has full Lebesgue measure if the sum $\sum_{j\ge 1} \Psi(j)$ diverges.
\end{thm}

\begin{rem}
As pointed out by the referee, this divergence statement holds even if the set $\mathcal{S}$ is \emph{regular} without being \emph{optimal}. See~\cite{beresbernikdodizurich} for further details.
\end{rem}

Let 
\begin{equation}\label{defiensS}
\mathcal{S}:=(0,a)\cap \left\{\frac{am+r}{bn+s}\right\}_{m,n\ge 0}.
\end{equation} 
The goal is to prove that $\mathcal{S}$ is an optimal regular system in the interval $E:=(0,a)$. Here, the elements of $\mathcal{S}$ are ordered by increasing denominator and, for two elements of $\mathcal{S}$ with the same denominator, by increasing numerator in such a way that the divergence part of Theorem~\ref{alakhint} will follow at once from Theorem~\ref{thmberesnoptiregusyst}.

It is not straightforward clear to the author that the optimal regularity of $\mathcal{S}$ in $E$ can be obtained in the same way as the optimal regularity of the rationals in the unit interval as established by Bugeaud in Proposition 5.3 of~\cite{bugeaudlivre1}~: indeed, Bugeaud's argument strongly rests on considerations of length combined with the use of Dirichlet's theorem applied to each irrational in the unit interval. In this case however, it follows from Corollary~\ref{resultatprincikhintchiunif} that a Dirichlet type result is satisfied by almost no irrational if $r\neq 0$ or $s\neq 0$.

In order to establish the optimal regularity of $\mathcal{S}$ with respect to $E$, two preliminary lemmas are first required. For the classical results related to some arithmetical functions mentioned in the proofs, see for instance~\cite{hw}.

\begin{lem}\label{lemme1sommeeulerarith}
Let $q\ge 1$ be an integer such that $\gcd(a,r,q)=1$. 

Then $$\sum_{\underset{\gcd(am+r,q)=1}{0\le am+r\le x}}1 \, = \, x\,\frac{\gcd(q,a)}{qa} \,\varphi\!\left(\frac{q}{\gcd(q,a)}\right) + O\left(2^{\omega(q)}\right),$$ where $\varphi$ denotes Euler's totient function and $\omega(q)$ the number of distinct prime divisors of $q$ and where the implicit constant depends only on $a$.
\end{lem}

\begin{proof}
\sloppy Let $\mu(\,.\,)$ denote the M\"obius function. Since for any integer $n\ge 1$, $\sum_{d|n}\mu(d)$ equals 1 if $n=1$ and 0 otherwise, one gets, for $x\ge a$, denoting by $\lfloor\,.\,\rfloor$ the floor function,
\begin{align*}
\sum_{\underset{\gcd(am+r,q)=1}{0\le am+r\le x}}1 \, & = \, \sum_{0\le am+r\le x} \,\,\sum_{d|\gcd(am+r,q)} \mu(d)\\
& = \, \sum_{d|q}\mu(d) \sum_{\underset{am\equiv -r \imod{d}}{0\le m \le (x-r)/a}} 1\, =\, \sum_{\underset{\gcd(d,a)|r}{d|q}}\mu(d) \sum_{\underset{am\equiv -r \imod{d}}{0\le m \le (x-r)/a}} 1 \\
& = \, \sum_{\underset{\gcd(d,a)|r}{d|q}}\mu(d)  \left( \frac{1}{d}  \left\lfloor \frac{x-r}{a}\right\rfloor \gcd(d,a) + O\left(\gcd(d,a) \right) \right)\\
& =\, \frac{x}{a} \sum_{\underset{\gcd(d,a)|r}{d|q}}\frac{\mu(d)\gcd(d,a)}{d} + O\left(\sum_{\underset{\gcd(d,a)|r}{d|q}}\mu(d)\gcd(d,a)\right).
\end{align*}
Now, on the one hand, $$\left|\sum_{\underset{\gcd(d,a)|r}{d|q}} \mu(d)\gcd(d,a) \right|\, \le \, a\sum_{d|q}\left|\mu(d)\right| \, = \, a2^{\omega(q)},$$ which provides the error term in the conclusion of the lemma. On the other, any integer $d$ dividing $q$ can be written in a unique way in the form $d=kl$ with $k|\frac{q}{\gcd(q,a)}$ and $l|\gcd(q,a)$ with $\gcd(k,l)=1$. Therefore, from the multiplicativity of the M\"obius function, 
\begin{align*}
\sum_{\underset{\gcd(d,a)|r}{d|q}}\frac{\mu(d)\gcd(d,a)}{d} \, & = \, \sum_{\underset{l|\gcd(q,a),\, l|r}{k|\frac{q}{\gcd(q,a)}}}\frac{\mu(k)\mu(l)\gcd(l,a)}{kl}\\
&=\, \left(\sum_{k|\frac{q}{\gcd(q,a)}}\frac{\mu(k)}{k} \right)\left( \sum_{\underset{l|r}{l|\gcd(q,a)}}\frac{\mu(l)\gcd(l,a)}{l}\right)\\
&= \, \frac{\varphi\left(\frac{q}{\gcd(q,a)}\right)}{\frac{q}{\gcd(q,a)}}\left( \sum_{l|\gcd(q,a,r)}\mu(l)\right) \, = \, \varphi\left(\frac{q}{\gcd(q,a)}\right)\,\frac{\gcd(q,a)}{q},
\end{align*}
where the second last equation follows from the well-known fact that 
\begin{equation}\label{classikmobiuseuler}
\sum_{k|q'}\frac{\mu(k)}{k} = \frac{\varphi(q')}{q'}
\end{equation} 
for all $q'\ge1$ and the last equation from the assumption that $\gcd(q,a,r)=1$. This completes the proof.
\end{proof}

The second lemma generalizes the classical estimate $$\sum_{k=1}^{Q}\varphi(k) \, = \, \frac{Q^2}{2\zeta(2)} + O(Q\log Q),$$ a proof of which can for instance be found in~\cite{hw} (Theorem 330).

\begin{lem}\label{lem2sommeeulerarith}
Let $u\ge 1$  and $v\ge 0$ be integers and let $Q\ge u$ be a real number. 

Then $$\sum_{1\le uk+v \le Q} \varphi(uk+v) \, = \, C(u,v)Q^2 + O(Q\log Q),$$ where the implicit constant depends only on $u$ and $v$ and where $$C(u,v) = \frac{\varphi(\gcd(u,v))}{\gcd(u,v)}\left(2u\zeta(2)\prod_{\underset{\pi|u}{\pi \textrm{ prime}}} \left(1-\frac{1}{\pi^2}\right) \right)^{-1}.$$
\end{lem}

\begin{proof}
Assume that $Q=uk+v$ for some integer $k\ge 1$. It is clearly sufficient to establish the result in this case. Then, if $v\neq 0$,
\begin{equation}\label{sommeindiceulerarith}
\sum_{1\le ul+v \le Q} \varphi(ul+v)\, = \, \sum_{l=0}^k \varphi(ul+v)\, \underset{(\ref{classikmobiuseuler})}{=} \, \sum_{l=0}^k (ul+v)\sum_{d|(ul+v)}\frac{\mu(d)}{d}\cdotp
\end{equation}
If $v=0$, the last two sums should start with $l=1$. To avoid cumbersome notation, the proof will be given in the case $v\neq 0$ and the reader can easily check that it remains valid if $v=0$ up to very little modifications.

The relation $d|(ul+v)$ means that there exists $d'\in\Z$ such that $dd'-ul=v$. This last Diophantine equation is solvable in $(d', l)\in\Z^2$ if, and only if, $\delta:=\gcd(d,u)|v$, in which case any solution is of the form $$(d', l)\, = \, \left(d_0'+\frac{u}{\delta}t , -l_0+\frac{d}{\delta}t\right),$$ where $t\in\Z$ and $(d_0', -l_0)$ is a particular solution. Then the constraint $0\le l \le k$ amounts to the following one~: $\frac{\delta}{d}l_0\le t \le (k+l_0)\frac{\delta}{d}$. Thus, (\ref{sommeindiceulerarith}) becomes~:
\begin{align}\label{varphicontinued}
\sum_{l=0}^k\varphi(ul+v)\, &= \, \sum_{\underset{dd'\equiv v\imod{u}}{1\le dd'\le uk+v}} d'\mu(d) \nonumber\\
&= \, \sum_{\delta|\gcd(u,v)}\sum_{\underset{\gcd(d,u)=\delta}{1\le d \le uk+v}}\sum_{\frac{\delta l_0}{d}\le t \le (k+l_0)\frac{\delta}{d}} \left(d_0'+\frac{ut}{\delta}\right)\mu(d) \nonumber \\
& = \, \sum_{\delta|\gcd(u,v)}\sum_{\underset{\gcd(d,u)=\delta}{1\le d \le uk+v}} \mu(d)\left(\sum_{\frac{\delta l_0}{d}\le t \le (k+l_0)\frac{\delta}{d}} \left(d_0'+\frac{ut}{\delta}\right)\right) \nonumber \\
&= \, \sum_{\delta|\gcd(u,v)}\sum_{\underset{\gcd(d,u)=\delta}{1\le d \le uk+v}} \mu(d)\left(\frac{u\delta}{2}\left(\frac{k}{d} \right)^2 + O\left( \frac{k\delta}{d}\right)\right) \nonumber \\
&= \, \sum_{\delta|\gcd(u,v)}\left(\sum_{\underset{\gcd(d,u)=\delta}{1\le d \le uk+v}} \frac{\mu(d)}{d^2}\right)\frac{u\delta k^2}{2} + O\left( \sum_{\delta|\gcd(u,v)} k\delta \sum_{\underset{\gcd(d,u)=\delta}{1\le d \le uk+v}}\frac{1}{d}\right),
\end{align}
where the error term in this last equation is clearly $O\left(k\log k \right)$. Now, on the one hand, $$\sum_{\underset{\gcd(d,u)=\delta}{1\le d \le uk+v}} \frac{\mu(d)}{d^2} \, = \, \sum_{\underset{\gcd(d,u)=\delta}{d=1}}^{\infty} \frac{\mu(d)}{d^2} - \sum_{\underset{\gcd(d,u)=\delta}{d=uk+v+1}}^{\infty} \frac{\mu(d)}{d^2}$$ and, on the other, $$\left|\sum_{\underset{\gcd(d,u)=\delta}{d=uk+v+1}}^{\infty} \frac{\mu(d)}{d^2}\right|\, \le \, \sum_{d=uk+v+1}^{\infty}\frac{1}{d^2} \, = \, O\left(\frac{1}{k}\right),$$ hence, from~(\ref{varphicontinued}), 
\begin{align}\label{varphicontinuedbis}
\sum_{l=0}^k\varphi(ul+v)\, = \, \sum_{\delta|\gcd(u,v)}\left(\sum_{\underset{\gcd(d,u)=\delta}{d=1}}^{\infty} \frac{\mu(d)}{d^2}\right)\frac{u\delta k^2}{2} + O\left( k\log k\right).
\end{align}
Since $\frac{\mu(d)}{d^2}$ is a multiplicative function, the series appearing on the right--hand side of this equation can be simplified. Indeed, assume first that $\delta =1$. Then the expansion in the Euler product of the series under consideration gives
\begin{align*}
\sum_{\underset{\gcd(d,u)=1}{d=1}}^{\infty} \frac{\mu(d)}{d^2} \, & = \, \prod_{\underset{\gcd(\pi, u) = 1}{\pi \textrm{ prime}}}\left(1+\sum_{l=1}^{\infty}\frac{\mu\left(\pi^l\right)}{\pi^{2l}} \right) \\ 
& = \, \prod_{\underset{\gcd(\pi, u) = 1}{\pi \textrm{ prime}}}\left(1-\frac{1}{\pi^{2}} \right) \, = \, \left( \zeta(2)\prod_{\underset{\pi | u}{\pi \textrm{ prime}}}\left(1-\frac{1}{\pi^2} \right)\right)^{-1}.
\end{align*}
If, now, $\delta\ge 2$ is a divisor of $\gcd(u,v)$, let $d$ be a square--free integer such that $\gcd(d,u)=\delta$. Write $d=\delta d'$ in such a way that $d'$ is a square--free integer satisfying $\gcd(d', \delta)=1$ and so $\gcd(d', u)=1$. Then,
\begin{align*}
\sum_{\underset{\gcd(d,u)=\delta}{d=1}}^{\infty} \frac{\mu(d)}{d^2} \, = \, \sum_{\underset{\gcd(d',u)=1}{d'=1}}^{\infty}\frac{\mu(\delta d')}{(\delta d')^2} \, = \, \frac{\mu(\delta)}{\delta^2}\sum_{\underset{\gcd(d',u)=1}{d'=1}}^{\infty}\frac{\mu(d')}{(d')^2} \, = \, \frac{\mu(\delta)}{\delta^2}\left( \zeta(2)\prod_{\underset{\pi | u}{\pi \textrm{ prime}}}\left(1-\frac{1}{\pi^2} \right)\right)^{-1}.
\end{align*}
Setting $\mu(1)=1$ and combining this with~(\ref{varphicontinuedbis}), one gets, in the case where $Q=uk+v$ for some $k\ge 1$, $$\sum_{0\le ul+v \le Q} \varphi(ul+v)\, = \, (Q-v)^2\left( 2u\zeta(2)\prod_{\underset{\pi | u}{\pi \textrm{ prime}}}\left(1-\frac{1}{\pi^2} \right)\right)^{-1}\sum_{\delta|\gcd(u,v)}\frac{\mu(\delta)}{\delta}+ O\left( Q\log Q\right),$$ which completes the proof from~(\ref{classikmobiuseuler}).
\end{proof}

\begin{proof}[Completion of the proof of Theorem~\ref{alakhint}]
The optimal regularity of the subset $\mathcal{S'}$ of $\mathcal{S}$ (defined by~(\ref{defiensS})) made up of fractions of the form $(am+r)/(bn+s)$ satisfying $\gcd(am+r,bn+s) = \gcd(a,b,r,s)$ will now be established. It should be clear that it may be assumed, without loss of ge\-ne\-ra\-li\-ty, that $\gcd(a,b,r,s) =1$.

\paragraph{}
Let us first prove the existence of a subsequence of the sequence $(bn+s)_{n\ge 0}$ of the form $(un+v)_{n\ge 0}$ ($u,v\ge 0$ integers) such that $\gcd(un+v, a,r)=1$ for all $n\ge 1$ if $\delta:=\gcd(a,r)>1$. Under the assumption that $\gcd(\delta,b, s)=1$, a prime divisor $\pi$ of $\delta$ cannot divide both $b$ and $s$. It is therefore possible to fix an integer $n_{\pi}$ defined modulo $\pi$ such that $bn_{\pi}+s\not\equiv 0 \imod{\pi}$ (set for example $n_{\pi}\equiv (1-s)b^{-1}\imod{\pi}$ if $\gcd(\pi,b)=1$ and $n_{\pi}\equiv 1\imod{\pi}$ otherwise). From the Chinese remainder theorem, there exists an integer $n_0$, defined uniquely modulo $\prod_{\pi|\delta}\pi$ (the product is taken over prime numbers), such that $n_0\equiv n_{\pi}\imod{\pi}$ for all primes $\pi$ dividing $\delta$. Then, set $u:=b\prod_{\pi|\delta}\pi$ and $v:=bn_0+s$ in such a way that $(uk+v)_{k\ge 0} = \left( b\left( n_0+k\prod_{\pi|\delta}\pi\right)+s\right)_{k\ge 0}$. It is then clear that for any element $N$ of the sequence $(uk+v)_{k\ge 0}$, $\gcd(N,\delta)=1$ since for all prime divisor $\pi$ of $\delta$, $$N\equiv bn_0+s \equiv bn_{\pi}+s \not\equiv 0 \imod{\pi}.$$

\paragraph{}
Let $I=(\alpha, \beta)$ (with $\alpha <\beta$)  denote an open interval contained in $(0,a)$. Consider the set of all elements of the sequence $(un+v)_{n\ge 0}$ which lie in the interval $[Q/2, Q]$, where $Q\ge u$ is a real number. It follows from Lemma~\ref{lemme1sommeeulerarith} that, for a fixed $n\ge 1$, the number of integers $m\ge 0$ such that $\gcd(am+r,un+v)=1$ and $(am+r)/(un+v)\in I$ is $$\sum_{\underset{\gcd(am+r,q)=1}{\alpha q< am+r < \beta q}}1 \, = \, \lambda(I)\frac{\gcd(q,a)}{a}\varphi\left(\frac{q}{\gcd(q,a)}\right) + O(2^{\omega(q)}),$$ where  $q:=un+v$. From the well--known estimate $2^{\omega(q)}=o\left(q^{\epsilon}\right)$ valid for all $\epsilon >0$, for $Q$ (and so for $q\ge Q/2$) large enough depending only on $a$, $r$ and $\lambda(I)$, this last quantity is such that 
\begin{equation}\label{cardinalcopremierarith}
\sum_{\underset{\gcd(am+r,q)=1}{\alpha q< am+r < \beta q}}1 \, \ge \, \lambda(I)\frac{\varphi(q)}{2a},
\end{equation} 
where we used the fact that $\varphi\left(\frac{q}{\gcd(q,a)} \right)\ge \frac{\varphi(q)}{\gcd(q,a)}$.

Define now $\mathcal{S'}_Q(I)$ as the subset of $\mathcal{S'}$ made up of all those irreducible fractions in $I$ of the form $(am+r)/(un+v)$ and such that $Q/2 \le un+v \le Q$~: the distance between two distinct elements $(am+r)/(un+v)$ and $(am'+r)/(un'+v)$ of $\mathcal{S'}_Q(I)$ satisfies the inequality $$\left|\frac{am+r}{un+v} -\frac{am'+r}{un'+v} \right|\, \ge \, \frac{1}{(un+v)(un'+v)}\, \ge \, \frac{1}{Q^2}\cdotp$$ Moreover, it follows from~(\ref{cardinalcopremierarith}) that the cardinality $\# \mathcal{S'}_Q(I)$ of the set $\mathcal{S'}_Q(I)$ satisfies the estimate $$\# \mathcal{S'}_Q(I) \, \ge \, \frac{\lambda(I)}{2a}\sum_{Q/2\le un+v \le Q}\varphi(un+v).$$ Therefore, from Lemma~\ref{lem2sommeeulerarith}, for $Q$ large enough depending only on $a$, $u$, $v$ and $\lambda(I)$, $$\# \mathcal{S'}_Q(I) \, \ge \, \frac{C(u,v)}{4a}\lambda(I) Q^2.$$
Up to constants, $\lambda(I)Q^2$ elements of $\mathcal{S'}_Q(I)\subset \mathcal{S'}$ have been found in $I$ such that the gap between any two of them is $Q^{-2}$. Furthermore, from the indexing adopted for $\mathcal{S}$ (which is also used for $\mathcal{S'}$), it should be clear that the largest index of an element of $\mathcal{S'}_Q(I)$ is at most $aQ^2$. Since this holds true for all $Q$ large enough (depending only on $\mathcal{S'}$ and $I$),  it is easy to see that Definition~\ref{defioptiregulsyst} applies. 

This completes the proof of the optimal regularity of the subset of $\mathcal{S'}$ and so of Theorem~\ref{alakhint}.
\end{proof}

The Mass Transference Principle, due to S.Velani and V.Beresnevich, allows one to deduce Corollary~\ref{hausdmetric} from  Theorem~\ref{alakhint} without much difficulty. Here, the result of~\cite{masstransfprincipe} is not given in full generality but adapted to our purpose.

\begin{thm}[Mass Transference Principle]\label{masstransfprinci}
Let $\Omega$ be a compact interval in $\R$ with non--empty interior and let $t\in (0,1)$. Denote by $\left(J_i\right)_{i\ge 0}$ a sequence of intervals in $\Omega$ whose lengths tend to zero as $i$ tends to infinity. For any interval $J$ centered at $x\in\Omega$ with half--length $r$, denote by $J^t$ the interval centered at $x$ with half--length $r^t$. Assume furthermore that 
\begin{align}\label{lebesguemasstransfcondi}
\lambda \left( \limsup_{i\rightarrow \infty} J_i^t\right)\, = \, \lambda\left(\Omega\right).
\end{align}
Then $$\mathcal{H}^t\left(\limsup_{i\rightarrow \infty} J_i\right)\, = \, \mathcal{H}^t\left(\Omega\right)\, = \, \infty .$$
\end{thm} 

\begin{proof}[Proof of Corollary~\ref{hausdmetric} from Theorem~\ref{masstransfprinci}]
Let $t\in (0,1)$. If the sum $\sum_{n=1}^{\infty}n\Psi(bn+s)^t$ converges, a standard covering argument shows that $\mathcal{H}^t\left(\mathcal{K}\left(\Psi\right)\right) = 0$~: here again, details are left to the reader.

Assume now that the sum $\sum_{n=1}^{\infty}n\Psi(bn+s)^t$ diverges and recall that the set $\mathcal{K}\left(\Psi\right)$ has been restricted without loss of generality to the interval $(0,a)$. Set $\Omega=[0,a]$ in the assumptions of Theorem~\ref{masstransfprinci} and chose $(J_i)_{i\ge 0}$ as being the sequence of all those intervals contained in $\Omega$ centered at rationals of the form $(am+r)/(bn+s)$ with $\gcd(am+r,bn+s)=\gcd(a,b,r,s)$ and of length $2\Psi(bn+s)$. These intervals are indexed in the usual way (see after~(\ref{defiensS})). Then, condition~(\ref{lebesguemasstransfcondi}) is met from the divergence part of Theorem~\ref{alakhint}, so that applying Theorem~\ref{masstransfprinci} completes the proof.
\end{proof}

\section{Proofs of the results related to uniform approximation}\label{unifapproxpreuve}

This section is devoted to the proofs of Theorem~\ref{CNSapproxunifarithnonmetr} and Corollary~\ref{resultatprincikhintchiunif}. Throughout, conditions~(\ref{contraintesdebase}) will be strengthened in assuming, without loss of generality from the discussion held in the introduction, that 
\begin{equation}\label{contraintesdebasebis}
r\neq 0\quad \mbox{or} \quad s\neq 0.
\end{equation}
First, some auxiliary results, dealing mainly with properties of continued fraction expansions, are recalled.

\subsection{Some auxiliary results}\label{soussectionresultatsauxiliaires}

The next lemma collects some well-known properties of the continued fraction of an irrational. 

\begin{lem}\label{lemfractioncontinuepropr}
Let $\xi$ be an irrational number with partial quotients $(a_k)_{k\ge 0}$ and convergents $(p_k/q_k)_{k\ge 0}$. Set conventionally $p_{-1}=1$, $q_{-1}=0$, $p_0=a_0$ and $q_0=1$.

Then~:
\begin{itemize}
\item[1.] For any $k\ge0$. 
\begin{align}\label{lemdvlptfractcontinue1}
q_k p_{k-1}-p_kq_{k-1} = (-1)^k.
\end{align}
In particular, $p_k$ and $q_k$ are coprime.

\item[2.] The numerators and the denominators of the convergents of $\xi$ satisfy the recurrence relation 
\begin{align}\label{lemdvlptfractcontinue5}
p_k = a_kp_{k-1}+p_{k-2} \quad \mbox{and} \quad q_k = a_k q_{k-1}+q_{k-2}
\end{align}
valid for all $k\ge 1$.

\item[3.] For any $k\ge 0$, 
\begin{align}\label{lemdvlptfractcontinue2}
\frac{1}{a_{k+1}+1}\, < \, \frac{q_k}{q_{k+1}} = [0;a_{k+1},a_{k+2},\dots]\, < \, \frac{1}{a_{k+1}}\cdotp
\end{align}

\item[4.] For any $k\ge 1$, set $$\phi_{k}:=\frac{q_k\xi-p_k}{q_{k-1}\xi-p_{k-1}}\cdotp$$ Then, $\phi_k<0$,
\begin{align}\label{lemdvlptfractcontinue3}
1+a_{k+1}\phi_k  = \phi_k \phi_{k+1} \quad \mbox{and} \quad \left|\phi_k\right|<\frac{1}{a_{k+1}}\cdotp
\end{align}

\item[5.] For any integer $k\ge -1$, set $$\eta_k:=(-1)^k\left(q_k\xi-p_k\right)>0.$$ Then, the sequence $(\eta_k)_{k\ge -1}$ decreases and
\begin{align}\label{lemdvlptfractcontinue4}
\frac{1}{2}\, \le \, \frac{q_{k+1}}{q_k+q_{k+1}}\, \le \, \eta_k q_{k+1} \,\le\, 1
\end{align}
for all $k\ge -1$.

\item[6.] Let $k\ge 1$, $a_0\in\Z$ and $a_1, \dots, a_k \ge 1$ be integers. Let furthermore $E(a_0, a_1, \dots, a_k)$ denote the set of real numbers whose $k+1$ first partial quotients are $a_0, a_1, \dots, a_k$. Then, 
$$E(a_0, a_1, \dots, a_k) = \begin{cases}
\left[\frac{p_k}{q_k},\frac{p_k+p_{k-1}}{q_k+q_{k-1}}\right) & \mbox{if } k\mbox{ is even}\\
\left(\frac{p_k+p_{k-1}}{q_k+q_{k-1}}, \frac{p_k}{q_k}\right] & \mbox{if } k\mbox{ is odd},
\end{cases}$$
where $p_{k-1}/q_{k-1} = [a_0;a_1, \dots, a_{k-1}]$ and $p_k/q_k = [a_0;a_1, \dots, a_k]$. In particular,
\begin{align}\label{lemdvlptfractcontinue6}
\frac{1}{2q_k^2}\, \le \, \lambda\left(E(a_0, a_1, \dots, a_k)\right) = \frac{1}{q_k(q_k+q_{k-1})}\, \le \, \frac{1}{q_k^2}\cdotp 
\end{align}

\item[7.] For any $k\ge 1$, 
\begin{align}
\prod_{j=1}^{k}a_j \, \le \, &q_k\, \le \, \prod_{j=1}^{k}(a_j+1)\, \le \, 2^k\prod_{j=1}^k a_j, \label{lemdvlptfractcontinue7a}\\
(1+a_0a_1)\prod_{j=2}^{k}a_j \, \le \, &p_k\, \le \, (1+a_0a_1)\prod_{j=2}^{k}(a_j+1)\, \le \, 2^{k-1}(1+a_0a_1)\prod_{j=2}^k a_j . \label{lemdvlptfractcontinue7b}
\end{align}

\item[8.] Let $d,t,u\ge 1$ be integers and let $\bm{k}\in\N^d$. Denote by $E_{\bm{k}}^{(t)}$ the set of all those irrationals $\xi$ in the interval $[t,t+1]$ such that $(a_0(\xi), a_1(\xi), \dots, a_d(\xi)) = (t,\bm{k})$. Let $E_{\bm{k},u}^{(t)}$ denote the subset of $E_{\bm{k}}^{(t)}$ made up of all those irrationals $\xi$ such that $a_{d+1}(\xi) = u$. Then
\begin{align}\label{lemdvlptfractcontinue8}
\frac{1}{3u^2} \, < \, \frac{\lambda\left(E_{\bm{k},u}^{(t)}\right)}{\lambda\left(E_{\bm{k}}^{(t)}\right)}\, < \, \frac{2}{u^2}\cdotp
\end{align}
If $t=0$, the set $E_{\bm{k}}^{(t)}$ (resp.~$E_{\bm{k},u}^{(t)}$) will simply be denoted by $E_{\bm{k}}$ (resp.~by $E_{\bm{k},u}$).
\end{itemize}
\end{lem}

\begin{proof}
Inequalities~(\ref{lemdvlptfractcontinue7a}) and~(\ref{lemdvlptfractcontinue7b}) can easily be obtained by induction from relations~(\ref{lemdvlptfractcontinue5}). All the other results are standard. See for instance Chapter 1 of~\cite{casselsapproxdioph} or Chapter 1 of~\cite{bugeaudlivre1} for proofs.
\end{proof}

The following generalizes a result well--known in the case $a_1=a_2=1$. The proof, which is elementary, is left to the reader.

\begin{lem}\label{lemelementaire}
For $i=1,2$, let $a_i \ge 0, b_i\ge 0$ and $m_i\ge 1$ denote natural integers. Then, the system of equations 
$$
\left\{
    \begin{array}{ll}
        a_1 x \equiv b_1 \imod{m_1} \\
        a_2 x \equiv b_2 \imod{m_2}
    \end{array}
\right.
$$
admits a solution $x\in\Z$ if, and only if, $$\gcd(m_1,a_1)\,|\, b_1, \quad \gcd(m_2,a_2)\,|\, b_2 \quad \mbox{and} \quad \gcd(a_1m_2, a_2m_1)\, |\, a_1b_2 - a_2b_1.$$
\end{lem}

\subsection{Non--metrical point of view}\label{nonmetrunif}

It is remarkable that, in the case where $\xi$ is a badly approximable irrational number, the result of Theorem~\ref{CNSapproxunifarithnonmetr} can be generalized by proving that $\xi$ admits an inhomogeneous uniform $(a,b,r,s)$--approximation with exponent 1. In what follows, Bad denotes the set of badly approximable irrationals.

\begin{prop}\label{propunifinhomgbad}
Let $\xi\in \mbox{Bad}$ and $\alpha\in\R$.

Then, there exists a constant $c(\xi)>0$ such that, for all real numbers $Q\ge 2b$, there are integers $m$ and $n$ satisfying
$$\left|(bn+s)\xi - (am+r)+\alpha\right|\, \le \, \frac{c(\xi)}{Q} \quad \mbox{ and } \quad 0\le bn+s \le Q.$$
Furthermore, $c(\xi)=2ab(M+2)$ is an admissible value, where $M$ is an upper bound for the partial quotients of $b\xi/a$.
\end{prop}

Proposition~\ref{propunifinhomgbad} will follow without much difficulty from the \emph{three distance theorem}, also called in the literature \emph{the Steinhaus theorem}, \emph{the three length}, \emph{three gap} or \emph{three step theorem}. The latter states that, for any positive integer $Q$ and for any irrational $\xi$, the points $\left(\{i\xi\} \right)_{0\le i \le Q}$ partition the unit interval into $Q+1$ subintervals, the lengths of which take at most three values, one being the sum of the other two (here, $\{x\}$ denotes the fractional part of a real number $x$). The reader is referred to~\cite{thmdes3dist} for a complete survey on the topic and to the references therein for various proofs of the precise statement of the result given below. The latter uses the fact that, for any integer $Q\ge 1$, there exist unique integers $k\ge 1$, $p$ and $w$ such that 
\begin{align}\label{algogloutonexpress}
Q=pq_{k-1} + q_{k-2}+w \quad \mbox{ with } \quad 1\le p \le a_{k} \quad \mbox{ and } \quad 0\le w <q_{k-1},
\end{align}
where $(q_k)_{k\ge 0}$ is the sequence of the denominators of the convergents of a given irrational $\xi$. Such a decomposition can be obtained thanks to the greedy algorithm. The notation introduced in Lemma~\ref{lemfractioncontinuepropr} is kept in the statement of the three distance theorem, in particular see~(\ref{lemdvlptfractcontinue4}) for the definition of $\eta_k$.

\begin{thm}[The three distance theorem]\label{thm3dist}
Let $\xi$ be an irrational and let $Q\ge 1$ be a positive integer given in the form~(\ref{algogloutonexpress}).

Then, the unit interval is divided by the points $0$, $\{\xi\}$, $\{2\xi\}, \dots , \{Q\xi\}$ into $Q+1$ subintervals which satisfy the following conditions~:
\begin{itemize}
\item $Q+1-q_{k-1}$ of them have length $\eta_{k-1}$,
\item $w+1$ have length $\eta_{k-2}-p\eta_{k-1}$,
\item $q_{k-1}-(w+1)$ have length $\eta_{k-2}-(p-1)\eta_{k-1}$.
\end{itemize}
\end{thm}

\begin{rem}
As $\xi$ is irrational, the three lengths are distinct. The third length, which is the largest since it is the sum of the other two, does not always appear. The other two do always appear.
\end{rem}

\begin{proof}[Proof of Proposition~\ref{propunifinhomgbad} from Theorem~\ref{thm3dist}]
In this proof, $(a_k)_{k\ge 0}$ (resp.~$(p_k/q_k)_{k\ge 0}$) refers to the sequence of the partial quotients (resp.~of the convergents) of the irrational $b\xi /a$, where $\xi\in \mbox{Bad}$. The integer $M$ shall denote an upper bound for the sequence $(a_k)_{k\ge 0}$.

Let $Q\ge 1$ be an integer and let $k\ge 1$, $p$ and $w$ be integers as given by~(\ref{algogloutonexpress}). From Theorem~\ref{thm3dist}, the unit interval is partitioned by the numbers $\left(\{ib\xi/a\} \right)_{0\le i \le Q}$ into $Q+1$ subintervals of lengths at most $\eta_{k-2}$. Modulo $a$ this is saying that the point $s\xi - r+\alpha$ lies within a distance $a\eta_{k-2}/2$ from $b\xi n$ for some integer $n$ in the interval $\llbracket 0,Q \rrbracket$. In other words, there exist $m\in\Z$ and $n\in \llbracket 0,Q \rrbracket$ such that $$\left|\left(b\xi n - am\right) + \left(s\xi-r+\alpha\right)\right|\, \le \, \frac{a\eta_{k-2}}{2} \quad \mbox{ and } \quad 0\,\le \, bn+s \, \le \, bQ+s\, \le\, 2bQ,$$
whence 
\begin{align*}
Q\left|\xi\left(bn+s\right) -\left(am+r\right)+\alpha \right|\, & \le \, \frac{a}{2}Q\eta_{k-2}\\
&\underset{(\ref{lemdvlptfractcontinue4})\& (\ref{algogloutonexpress})}{\le} \, \frac{a}{2} \left(a_k+2\right)\\
&\le \,\frac{a}{2} (M+2).
\end{align*}
Assume now that $Q\ge 2b$ is a real number and set $Q'=\lfloor Q/(2b) \rfloor \ge 1$~: from what precedes, there exist integers $m$ and $n$ such that $0\le bn+s\le 2bQ' \le Q$ and $$Q\left|\xi\left(bn+s\right) -\left(am+r\right)+\alpha \right|\, \le \, \frac{Q}{Q'}\frac{a}{2}(M+2)\, \le \, \left(1+\frac{1}{Q'} \right)ab(M+2)\, \le \, 2ab(M+2).$$
This completes the proof of Proposition~\ref{propunifinhomgbad}.
\end{proof}

\paragraph{}
The rest of this subsection is devoted to the proof of Theorem~\ref{CNSapproxunifarithnonmetr}, where the notation introduced in Lemma~\ref{lemfractioncontinuepropr} will be systematically used with respect to a fixed $\xi\in\R\backslash\Q$. To this end, first notice that, given $m,n\in\Z$ and $k\ge 1$, it follows from~(\ref{lemdvlptfractcontinue1}) that there exists a unique pair $(u,v)\in\Z^2$, given by 
\begin{align}
u\, &=\, (-1)^{k-1} \left[(am+r)q_{k-1}-(bn+s)p_{k-1}\right] \label{uv1},\\
v\, &=\, (-1)^{k-1} \left[(bn+s)p_{k-2} - (am+r)q_{k-2} \right], \label{uv2}
\end{align}
such that 
\begin{align}
am+r\, & = \, up_{k-2}+vp_{k-1} \label{mn1},\\
bn+s\, & = \, uq_{k-2}+vq_{k-1} \label{mn2}.
\end{align}
Furthermore, in this case, on noticing that  $$\frac{\left|\xi(bn+s)-(am+r)\right|}{q_{k-1}\left|q_{k-2}\xi - p_{k-2}\right|}\, = \, \frac{1}{q_{k-1}}\cdotp \frac{\left|\xi\left( uq_{k-2}+vq_{k-1}\right)-\left(up_{k-2}+vp_{k-1}\right)\right|}{\left|q_{k-2}\xi-p_{k-2}\right|}\, = \, \frac{1}{q_{k-1}} \left|u+v\phi_{k-1}\right|,$$ where $\phi_{k-1}$ has been defined in~(\ref{lemdvlptfractcontinue3}), inequalities~(\ref{lemdvlptfractcontinue4}) imply that
\begin{equation}\label{uvnmkQ}
\frac{1}{2q_{k-1}} \left|u+v\phi_{k-1}\right|\, < \, \left|\xi(bn+s)-(am+r)\right| \, < \, \frac{1}{q_{k-1}} \left|u+v\phi_{k-1}\right|.
\end{equation}

\begin{proof}[Proof of the necessary part of Theorem~\ref{CNSapproxunifarithnonmetr}]
Assume that there exists a strictly in\-crea\-sing sequence $(k_l)_{l\ge 0}$ of natural integers such that conditions~(\ref{conditionfondaapproxunif}) are not met for the index $k_l$ and such that 
\begin{equation}\label{preuvethmCNS0}
\lim_{l\rightarrow \infty} \, \frac{a_{k_l}}{\widetilde{\Psi}\left(q_{k_l}\right)} \, = \, \infty .
\end{equation} 
For a contradiction, assume that there is a $Q_0\ge1$ such that for each integer $Q\ge Q_0$, there exist $m,n\in\Z$ satisfying 
\begin{equation}\label{preuvethmCNS1}
0\le bn+s \le Q \quad \mbox{ and }\quad \Psi(Q)^{-1}\left|\xi(bn+s)-(am+r)\right|\,\le \, c(\xi)
\end{equation}
for some constant $c(\xi)>0$. Assuming without loss of generality that $k_0$ has been chosen in such a way that $q_{k_0}/2\ge Q_0$, set furthermore, for all $l\ge 0$, $$Q_{k_l}:=\left\lceil \frac{q_{k_l}}{2} \right\rceil,$$ where $\lceil \, . \, \rceil$ denotes the ceiling function. Let $l\ge 0$ and $m$ and $n$ be integers verifying~(\ref{preuvethmCNS1}) for the integer $Q_{k_l}$.
It then follows from~(\ref{uv2}) that 
\begin{align*}
\left|v\right|\, = \, q_{k_l-2}\left|(am+r)-(bn+s)\frac{p_{k_l-2}}{q_{k_l-2}}\right| \, & \underset{(\ref{preuvethmCNS1})}{\le} \, q_{k_l-2}\left|\xi(bn+s)-(am+r)\right| + q_{k_l-2} \left|\xi - \frac{p_{k_l-2}}{q_{k_l-2}}\right| Q_{k_l}\\
&\underset{(\ref{lemdvlptfractcontinue4})\&(\ref{preuvethmCNS1})}{\le} \, c(\xi) q_{k_l-2}\Psi\!\left( \left\lceil \frac{q_{k_l}}{2}\right\rceil\right) + \frac{\left\lceil q_{k_l}/2\right\rceil}{q_{k_l-1}}\\
&\underset{(\Psi \textrm{ decreases})}{\le} \, c(\xi) q_{k_l-2}\Psi\!\left(\frac{q_{k_l}}{2}\right) + \frac{ q_{k_l}/2 +1}{q_{k_l-1}}\\
&\underset{(\ref{lemdvlptfractcontinue2})}{\le} \, c(\xi) \frac{q_{k_l-2}}{q_{k_l}/2} \widetilde{\Psi}\!\left(\frac{q_{k_l}}{2}\right) +\frac{1}{2} (a_{k_l}+1) + \frac{1}{q_{k_l-1}} \\
& \underset{(\ref{conditiondecroissancepsi})}{=} \, O\left( \widetilde{\Psi}\left(q_{k_l}\right)\right)+\frac{a_{k_l}}{2} + O\left(1\right).
\end{align*}
Therefore, $$\left|v\phi_{{k_l}-1}\right| \, \underset{(\ref{lemdvlptfractcontinue3})}{\le} \, \frac{O\left( \widetilde{\Psi}\left(q_{k_l}\right)\right)+ a_{k_l}/2 + O\left(1\right)}{a_{k_l}} \, \underset{(\ref{preuvethmCNS0})}{=} \, 1/2+o(1).$$ On the other hand, the integer $u$ cannot equal zero in the representations~(\ref{mn1}) and~(\ref{mn2})~: indeed, this would otherwise contradict the fact that conditions~(\ref{conditionfondaapproxunif}) are not met for the index $k_l$ from Lemma~\ref{lemelementaire}.

Thus, since $|u|\ge 1$, one gets~:
\begin{align*}
\Psi\left(Q_{k_l}\right)^{-1}\left|\xi(bn+s)-(am+r)\right|\, & \underset{(\Psi \textrm{ decreases})}{\ge} \, \Psi\left(\frac{q_{k_l}}{2}\right)^{-1}\left|\xi(bn+s)-(am+r)\right|\\
& \underset{(\ref{uvnmkQ})}{\ge} \, \widetilde{\Psi}\left(\frac{q_{k_l}}{2}\right)^{-1} \frac{q_{k_l}/2}{2q_{k_l-1}}\left(\left|u\right| - \left|v\phi_{{k_l}-1}\right|\right)\\
& \underset{(\ref{conditiondecroissancepsi})}{\ge} \, \left(\kappa\widetilde{\Psi}\left(q_{k_l}\right)\right)^{-1} \frac{q_{k_l}}{4q_{k_l-1}}\left(1 - \frac{1}{2}+o(1)\right)\\
& \underset{(\ref{lemdvlptfractcontinue2})}{\ge} \, \frac{1}{4\kappa}\widetilde{\Psi}\left(q_{k_l}\right)^{-1} a_{k_l}\left(\frac{1}{2}+o(1)\right),
\end{align*}
which, from~(\ref{preuvethmCNS0}), contradicts~(\ref{preuvethmCNS1}) for $l$ large enough and completes the proof.
\end{proof}

The proof of the sufficiency of the conditions attached to~(\ref{conditionfondaapproxunif}) in Theorem~\ref{CNSapproxunifarithnonmetr} is more involved.

\begin{proof}[Proof of the sufficient part of Theorem~\ref{CNSapproxunifarithnonmetr}]
Assume that $Q\ge 1$ is an integer written in the form~(\ref{algogloutonexpress}) for some integers $k\ge1$, $p$ and $w$. From~(\ref{mn1}), (\ref{mn2}) and~(\ref{uvnmkQ}), the problem comes down to proving the existence of integers $u$ and $v$ (and so $m$ and $n$) such that an upper bound depending only on $a$,$b$, $\xi$ and $\Psi$ might be found for the quantity $$\widetilde{\Psi}(Q)^{-1}\frac{Q}{q_{k-1}}\left|u+v\phi_{k-1}\right|$$ under the constraint $0\le bn+s\le Q$.

To this end, set $d:=\gcd(bp_{k-1}, aq_{k-1})$ and consider the unique integer $u$ lying in the interval $\llbracket 0, d-1\rrbracket$ which satisfies the congruence 
\begin{equation}\label{preuveCSthm30}
u\equiv (-1)^{k-1}\left(rp_{k-1}-sq_{k-1}\right) \imod{d}.
\end{equation}
Since $p_{k-1}$ and $q_{k-1}$ are coprime, one has in fact 
\begin{equation}\label{preuveCSthm31}
0\, \le \, u\, \le \, ab.
\end{equation}
Furthermore, under these assumptions, the equation 
\begin{equation}\label{preuveCSthm32}
u - (-1)^{k-1}\left(rp_{k-1}-sq_{k-1}\right) \, =\, (-1)^{k-1}\left(aq_{k-1}m-bp_{k-1}n \right)
\end{equation} 
is solvable in $(m,n)\in\Z^2$ and the set of all solutions can be written in the form $$\left(m_0-(-1)^{k-1}\frac{bp_{k-1}}{d}h\, ;\, n_0-(-1)^{k-1}\frac{aq_{k-1}}{d}h \right),$$ where $h\in\Z$ and $(m_0, n_0)\in\Z^2$ is a particular solution. This implies that there exists a unique pair $(m,n)\in\Z^2$ satisfying~(\ref{preuveCSthm32}) with the additional constraint $0\le n < aq_{k-1}/d$. For such a pair, it should be clear that
\begin{equation}\label{preuveCSthm33}
0\, \le \, bn+s \, \le \, baq_{k-1}+s \, \underset{(\ref{algogloutonexpress})}{\le} \, \min\{abQ, 2abq_{k-1}\}.
\end{equation} 
 On the other hand, eliminating $am+r$ in equations~(\ref{uv1}) and~(\ref{uv2}) gives $$v=(-1)^{k-1}\left(\frac{p_{k-2}}{q_{k-2}}-\frac{p_{k-1}}{q_{k-1}}\right)q_{k-2}(bn+s) - u\frac{q_{k-2}}{q_{k-1}},$$ hence $$\left|u+v\phi_{k-1} \right|\, < \, \left|u\right|. \left|1-\phi_{k-1}\frac{q_{k-2}}{q_{k-1}}\right| + \left|bn+s\right|. \left|\phi_{k-1}\right|.\left|\frac{p_{k-2}}{q_{k-2}}-\frac{p_{k-1}}{q_{k-1}}\right|q_{k-2}.$$ Taking into account~(\ref{lemdvlptfractcontinue1}), (\ref{lemdvlptfractcontinue3}), (\ref{preuveCSthm31}) and~(\ref{preuveCSthm33}), this leads to the inequality 
\begin{equation}\label{preuveCSthm34}
\left|u+v\phi_{k-1}\right|\, < \, 2ab + 2ab \,=\, 4ab.
\end{equation} 
Since the function $\Psi$ is non--increasing and since $Q\le q_k+q_{k-1} \le 2q_k$ from~(\ref{algogloutonexpress}), for such a choice of the integers $u$ and $v$ (and so, of the integers $m$ and $n$), one has~:
\begin{align}
\Psi(Q)^{-1} \left| \xi (bn+s)- (am+r)\right| \, &\le \, \Psi(2q_k)^{-1} \left| \xi (bn+s)- (am+r)\right| \nonumber\\ 
&= \, 2q_k\widetilde{\Psi}(2q_k)^{-1} \left| \xi (bn+s)- (am+r)\right| \nonumber\\
& \underset{(\ref{uvnmkQ})}{\le} \, \frac{2q_k}{q_{k-1}}\widetilde{\Psi}(2q_k)^{-1} \left| u+v\phi_{k-1}\right| \nonumber\\
& \underset{(\ref{lemdvlptfractcontinue2})}{\le } \, 2\left(a_k+1\right)\widetilde{\Psi}(2q_k)^{-1} \left| u+v\phi_{k-1}\right| \nonumber\\
& \underset{(\ref{preuveCSthm34})}{\le } \, 16 a b\, a_k \widetilde{\Psi}(2q_k)^{-1} \, \underset{(\ref{conditiondecroissancepsi})}{\le} \, 16 \kappa a b \,a_k \widetilde{\Psi}(q_k)^{-1}. \label{16kappaabM}
\end{align}
Now, if conditions~(\ref{conditionfondaapproxunif}) are not satisfied, this last quantity is less than $16\kappa ab M$ for some integer $M\ge 1$. If, however, conditions~(\ref{conditionfondaapproxunif}) are met, instead of choosing $u$ according to the constraints~(\ref{preuveCSthm30}) and~(\ref{preuveCSthm31}), set $u=0$. Then, from Lemma~\ref{lemelementaire}, there exist $v\in\llbracket 0, ab\rrbracket$ and $(m,n)\in\Z^2$ such that $$vp_{k-1} = am+r \quad \mbox{ and }\quad vq_{k-1} = bn+s,$$ in which case $0\le bn+s\le abQ$ and, repeating the above calculations, 
\begin{align*}
\Psi(Q)^{-1} \left| \xi (bn+s)- (am+r)\right| \, &\le \, 2\left(a_k+\frac{1}{q_{k-1}}\right)\widetilde{\Psi}(2q_k)^{-1} \left| v\phi_{k-1}\right|\\
&\underset{(\ref{conditiondecroissancepsi})}{\le} \, 4\kappa ab\, a_k \widetilde{\Psi}(q_k)^{-1} \left|\phi_{k-1}\right|\\
&\underset{(\ref{conditiondecroissancepsi})\&(\ref{lemdvlptfractcontinue3})}{\le} \, 4\kappa ab \gamma^{-1}.
\end{align*}

Thus, it has been proved that for all integers $Q\ge 1$, there exist $(m,n)\in\Z^2$ such that $\Psi(Q)^{-1}\left|\xi(bn+s)-(am+r) \right|\le 4\kappa ab\max\{4M, \gamma^{-1} \}$ under the constraint $0\le bn+s \le abQ$. Assume now that $Q\ge ab$ is any real number and set $Q':=\left\lfloor Q/(ab)\right\rfloor \ge 1$. Then there exist integers $m$ and $n$ such that $0\le bn+s\le abQ' \le Q$ and 
\begin{align*}
\Psi(Q)^{-1}\left|\xi(bn+s)-(am+r) \right|\, & \le\, 4\kappa ab\max\{4M, \gamma^{-1}\}\frac{\Psi(Q')}{\Psi(Q)}\\
& \underset{(\Psi \textrm{ decreases})}{\le}\, 4\kappa ab\max\{4M, \gamma^{-1}\}\frac{\Psi(Q')}{\Psi(ab(Q'+1))}\\
&= 4\kappa (ab)^2\max\{4M, \gamma^{-1}\}\left(1+\frac{1}{Q'}\right)\frac{\widetilde{\Psi}(Q')}{\widetilde{\Psi}(ab(Q'+1))}\\
&\underset{(\ref{conditiondecroissancepsi})}{\le} \, 8\kappa\eta (ab)^2\max\{4M, \gamma^{-1}\}.
\end{align*}
This completes the proof of Theorem~\ref{CNSapproxunifarithnonmetr}.
\end{proof}

\begin{rem}\label{etsikgek_0}
If there exist integers $M\ge 1$ and $k_0\ge 1$ such that, for all $k\ge k_0$, the inequality $a_k\le M \widetilde{\Psi}(q_k)$ holds true whenever conditions~(\ref{conditionfondaapproxunif}) are not met, then the conclusion of Theorem~\ref{CNSapproxunifarithnonmetr} remains true upon choosing $Q_0=ab(q_{k_0-1}+q_{k_0-2})$ (which quantity equals $ab$ when $k_0=1$).

Indeed, the previous proof applies with the exception that the upper bound $16\kappa abM$ used for the right--hand side of~(\ref{16kappaabM}) when conditions~(\ref{conditionfondaapproxunif}) are not satisfied is only valid if $k\ge k_0$. From the uniqueness of the decomposition~(\ref{algogloutonexpress}), this imposes the condition $Q\ge q_{k_0-1}+q_{k_0-2}$. Therefore, in the last step of the proof, the integer $Q':=\left\lfloor Q/(ab)\right\rfloor$ will be asked to be bigger than $q_{k_0-1}+q_{k_0-2}$, hence the choice of $Q_0$ in this case.
\end{rem}

An interesting question related to Theorem~\ref{CNSapproxunifarithnonmetr} is to study the size of the set of well--approximable numbers admitting a Dirichlet type approximation in the context of $(a,b,r,s)$--approximation. In this respect, the following conjecture seems of relevance.

\begin{conj}
The set of real numbers which are not in Bad and which admit a uniform $(a,b,r,s)$--approximation with exponent 1 has full Hausdorff dimension.
\end{conj}

Obviously, this conjecture is trivially true if $r=s=0$ from the discussion held in the introduction. On the other hand, the construction of a Cantor set to prove the conjecture seems easier in the case when $\gcd(a,b)=1$~: this is because, if one can ensure that the denominators $q_k$ (resp.~the numerators $p_k$) of the convergents of an irrational $\xi$ are all coprime to $b$ (resp.~to $a$), then conditions~(\ref{conditionfondaapproxunif}) always hold true. However, if $\gcd(a,b)>1$, the third condition in~(\ref{conditionfondaapproxunif}) turns out to be more delicate to deal with.

\subsection{Metrical point of view}

This subsection is devoted to the proof of Corollary~\ref{resultatprincikhintchiunif}. Throughout, the result will be established in the case where $s\neq 0$~: it is not difficult to verify that the reasoning below can easily be modified to obtain the same result in the case where $s=0$ and $r\neq 0$ working with the numerators of the convergents rather than with the denominators.

Consider a function $\Psi$ satisfying the assumptions of Corollary~\ref{resultatprincikhintchiunif}. Since $\widetilde{\Psi}$ is non--decreasing, it is clear that conditions~(\ref{conditiondecroissancepsi}) are satisfied, so that the conclusions of Theorem~\ref{CNSapproxunifarithnonmetr} hold true. In what follows, the metrical result of Corollary~\ref{resultatprincikhintchiunif} will be proved for the set $\mathcal{U}(\Psi)\cap [0,a]$ which, for the sake of simplicity, shall still be denoted by $\mathcal{U}(\Psi)$~: it should be clear that this suffices to establish Corollary~\ref{resultatprincikhintchiunif} in full generality.

More precisely, it will be shown that~:
\begin{itemize}
\item[a)] if the sum $\sum_{Q\ge 1}\frac{1}{Q^2\Psi(Q)}$ converges, then, for almost all $\xi\in [0,1]\backslash\Q$, there exists an integer $k_0(\xi)\ge 1$ such that, for all $k\ge k_0(\xi)$, $a_k(\xi)\le \tau\widetilde{\Psi}\left(q_k(\xi)\right)$. Hence the fact that $\lambda\left(\mathcal{U}\left(\Psi\right)\right)=1$ will follow from Theorem~\ref{CNSapproxunifarithnonmetr} and Remark~\ref{etsikgek_0} for a suitable choice of the parameter $\tau>0$.
\item[b)] if the sum $\sum_{Q\ge 1}\frac{1}{Q^2\Psi(Q)}$ diverges, then the set of $\xi\in [0,1]\backslash\Q$ such that, for all integer $M\ge 1$, there exist infinitely many indices $k\ge 1$ such that $b|q_{k-1}(\xi)$ and $a_{k}(\xi)\ge M \widetilde{\Psi}\left(q_k(\xi)\right)$ has strictly positive measure. By virtue of~(\ref{cxiexpression}) in Theorem~\ref{CNSapproxunifarithnonmetr}, an element $\xi$ belonging to the latter set cannot belong to the set
\begin{equation}\label{Vpsi}
\mathcal{V}\left(\Psi\right):=\bigcup_{c\ge 1}\mathcal{U}\left(c\Psi\right)
\end{equation}
whose complement has therefore strictly positive measure. Showing that $\mathcal{V}\left(\Psi\right)$ has either zero or full measure will then complete the proof in this case also.
\end{itemize}

The proof of Corollary~\ref{resultatprincikhintchiunif} requires a Borel--Berstein type technical lemma on continued fractions.

\subsubsection{A Borel--Berstein type technical lemma on continued fractions}

The classical theorem of Borel--Bernstein on continued fractions states that, given a sequence $\left(u_k\right)_{k\ge  1}$ of positive integers, if the sum $\sum_{k\ge 1}u_k^{-1}$ diverges, then, for almost all $\xi:=[0; a_1, a_2, \dots]$ in $[0,1)$, there exist infinitely many integers $k\ge 1$ such that $a_k\ge u_k$. Further, if the sum converges, then, for almost all $\xi:=[0; a_1, a_2, \dots]$ in $[0,1)$, there exist only a finite number of integers $k\ge 1$ such that $a_k \ge u_k$ (see for instance Theorem 1.11 in~\cite{bugeaudlivre1} for a proof). The following generalizes the Borel--Bernstein theorem and is the key step in proving Corollary~\ref{resultatprincikhintchiunif}.

\begin{lem}\label{lemborelbernstein}
Let $A\ge 1$ and $d\ge 1$ be integers. Denote by $\bm{f}:=\left( \bm{f}_k\right)_{k\ge 1}$ a sequence of functions such that, for every $k\ge 1$, the function $$\bm{f}_k~: \xi\in [0,1]\backslash\Q \mapsto \bm{f}_k(\xi)\in \llbracket 1, A\rrbracket^d$$ is measurable. Assume furthermore that $\varphi:=\left(\varphi_k \right)_{k\ge 1}$ is a sequence of positive integers for which there exists an integer $c\ge d+1$ such that the two series $\sum_{k=0}^{\infty}\varphi_k$ and $\sum_{k=0}^{\infty}\varphi_{ck}$ converge (resp.~diverge) simultaneously. For any $k\ge 1$, define the sets 
\begin{equation*}
E_k^d\left(\bm{f}_k, \varphi_k\right)\,:=\, \left\{ \xi\in [0,1]\backslash\Q \; : \; (a_k, a_{k+1}, \dots, a_{k+d-1})=\bm{f}_k\left(\xi\right)\, \mbox{ and }\, a_{k+d}\ge \varphi_k\right\}
\end{equation*}
and $$\mathcal{S}^d\left(\bm{f}, \varphi \right)\, := \, \limsup_{k\rightarrow \infty} E_k^d\left(\bm{f}_k, \varphi_k\right).$$

Then
$$
\lambda\left(\mathcal{S}^d\left(\bm{f}, \varphi \right)\right)  \begin{cases}
= 0 & \mbox{if } \sum_{k=0}^{\infty}\varphi_k^{-1}\, < \, \infty,\\
\ge \frac{\log 2}{4\left( 2(2A)^d\right)^4} & \mbox{if } \sum_{k=0}^{\infty}\varphi_k^{-1}\, = \, \infty.
\end{cases}
$$
\end{lem}

\begin{rem}
The assumption of the existence of the constant $c$ is a restriction of a technical nature~: as will be clear from the proof, it plays no role but to ensure that for an element $x$ lying in the intersection $E_{ck}^d\left(\bm{f}_{ck}, \varphi_{ck}\right) \cap E_{cl}^d\left(\bm{f}_{cl}, \varphi_{cl}\right)$, where $k$ and $l$ are two distinct positive integers, the two blocks $\left(a_{ck}(x), \dots, a_{ck+d-1}(x), a_{ck+d}(x)\right)$ and $\left(a_{cl}(x), \dots, a_{cl+d-1}(x), a_{cl+d}(x)\right)$ do not overlap.
\end{rem}

\begin{nota}
In order to prove Lemma~\ref{lemborelbernstein}, the notation introduced in the statement of the result is kept. Two additional sets are defined as follows~: given positive integers $k$, $d$ and $\beta$, given $\bm{\alpha}\in \N^d$, let $$E_k^d\left(\bm{\alpha}, \beta\right)\,:=\, \left\{ \xi\in [0,1]\backslash\Q \; : \; (a_k, a_{k+1}, \dots, a_{k+d-1})=\bm{\alpha}\, \mbox{ and }\, a_{k+d}\ge \beta\right\}$$ and $$\widetilde{E_k^d}\left(\bm{\alpha}, \beta\right)\,:=\, \left\{ \xi\in [0,1]\backslash\Q \; : \; (a_k, a_{k+1}, \dots, a_{k+d-1})=\bm{\alpha}\, \mbox{ and }\, a_{k+d}= \beta\right\}.$$
\end{nota}

\begin{proof}[Proof of the convergent part of Lemma~\ref{lemborelbernstein}] The convergent part of Lemma~\ref{lemborelbernstein} follows in the same way as the convergent part of the theorem of Borel--Bernstein, which in turn is nothing but a consequence of the Borel--Cantelli lemma. Details are provided here for the sake of completeness.

Suppose $\sum_{k=0}^{\infty}\varphi_k^{-1}< \infty$ and let $E\left(\varphi_k\right):=\left\{\xi\in [0,1]\backslash\Q \; : \; a_k\ge \varphi_k \right\}$ ($k\ge 1$). From the uniqueness of the continued fraction expansion of an irrational, one gets for all $k\ge 1$, using point~(\ref{lemdvlptfractcontinue8}) from Lemma~\ref{lemfractioncontinuepropr}, $$\lambda\left(E\left(\varphi_{k+1}\right)\right) \, = \, \sum_{\bm{\alpha}\in\N^k}\sum_{u\ge \varphi_{k+1}}\lambda\left(E_{\bm{\alpha},u}\right) \, \underset{(\ref{lemdvlptfractcontinue8})}{\le}\, \sum_{\bm{\alpha}\in\N^k}\sum_{u\ge \varphi_{k+1}}\frac{2}{u^2}\lambda\left(E_{\bm{\alpha}}\right) \, = \, \sum_{u\ge \varphi_{k+1}}\frac{2}{u^2}\, \le \, \frac{2}{\varphi_{k+1}}\cdotp$$
Thus, the series $\sum_{k\ge 1}\lambda\left(E\left(\varphi_{k}\right)\right)$ converges. Since $\mathcal{S}^d\left(\bm{f}, \varphi \right) \subset \limsup_{k\ge 1} E\left(\varphi_{k}\right)$, the result follows from the Borel--Cantelli lemma.
\end{proof}

\begin{rem}\label{remborelberni}
The proof of the convergent part of Lemma~\ref{lemborelbernstein} is also valid if $d=0$ (the only defining condition of the set $E_k^0\left(\bm{f}_k,\varphi_k\right)$ is then that $a_{k}\ge \varphi_k$), in which case the integer $c$ in the assumptions can be taken as equal to 1.
\end{rem}

The proof of the divergence half of Lemma~\ref{lemborelbernstein} is more involved. The use of the Gauss measure $\mu$ will make it simpler. The latter is defined for any element $E$ of the Borel $\sigma$--algebra $\mathcal{B}_{[0,1]}$ of $[0,1]$ by the formula $$\mu\left(E\right):=\frac{1}{\log 2}\int_{E}\frac{\textrm{d}x}{1+x}\cdotp$$ It should be clear that 
\begin{equation}\label{lebesguegaussmesures}
\frac{\lambda}{2\log 2}\, \le \, \mu \, \le \, \frac{\lambda}{\log 2}\cdotp
\end{equation}
In particular, the Lebesgue measure $\lambda$ restricted to $[0,1]$ and the Gauss measure $\mu$ are mutually absolutely continuous and therefore have the same sets of full and null measure. Define furthermore the Gauss map $T$ as follows~: $$T~: x=[0;a_1,a_2, \dots] \, \in\, [0,1]\backslash\Q \, \mapsto \, \left\{\frac{1}{x}\right\} = [0; a_2, a_3, \dots] \, \in\, [0,1]\backslash\Q ,$$ where $\{x\}$ denotes the fractional part of a real number $x$. It is a well--known fact (see for example Theorem 3.7 in~\cite{einsiedlerward}) that the system $\left(T,\mu, \mathcal{B}_{[0,1]} \right)$ is ergodic in $[0,1]$ and so that $\mu$ is $T$ invariant.

Two classical lemmas, which will be used in the proof of the divergent part of Lemma~\ref{lemborelbernstein}, are now introduced. The first one is essentially due to Khintchine (see e.g.~\cite{khintchi1} or~\cite{khintchi2}).

\begin{lem}\label{lemme1aborelberni}
Let $\bm{\alpha}\in\N^d$, where $d\ge 1$. Denote by $E_{\bm{\alpha}}$ the set $$E_{\bm{\alpha}}:= \left\{\xi\in [0,1)\; : \; (a_1, \dots, a_d) = \bm{\alpha}\right\}.$$ Let $F$ be a $\mu$--measurable set in $[0,1]$.

Then, there exists an absolute constant $\theta\in (0,1)$ such that for any $k\ge 0$, $$\mu\left( E_{\bm{\alpha}} \cap T^{-k-d}\left(F\right)\right)\, = \, \mu\left(E_{\bm{\alpha}}\right) \mu\left(F\right)\left(1+O\left(\theta^{\sqrt{k}}\right) \right).$$ The implicit constant in this last equation is also absolute.
\end{lem}

\begin{proof}
See~\cite{philippmetrical} for an explicit proof.
\end{proof}

The second lemma provides a partial converse to the Borel--Cantelli lemma.

\begin{lem}\label{reciproqueborelcantelli}
Let $\left(E_i\right)_{i\ge 0}$ be a sequence of $\mu$--measurable sets in $[0,1]$ such that $\sum_{i=0}^{\infty}\mu\left(E_i\right) = \infty$.

Then, $$\mu\left(\limsup_{i\rightarrow \infty}\, E_i\right) \, \ge \, \limsup_{i\rightarrow \infty} \left(\frac{\left(\sum_{k=1}^i \mu\left(E_k\right)\right)^2}{\sum_{1\le k,l\le i}\mu\left(E_k\cap E_l\right)}\right)\cdotp$$
\end{lem}

\begin{proof}
See e.g.~\cite{bugeaudlivre1}, p.125.
\end{proof}

\begin{proof}[Proof of the divergent part of Lemma~\ref{lemborelbernstein}] Suppose $\sum_{k=0}^{\infty}\varphi_k^{-1}= \infty$. The result will be established in four steps.

\paragraph{Step 1.} Given $k\ge 0$, the first step consists of finding a lower and an upper bound for $\mu\left(E_0^d\left(\bm{\alpha}, \varphi_k\right)\right)$ independently of $\bm{\alpha}\in \llbracket 1, A \rrbracket^d$. To this end, first notice that, from the uniqueness of the continued fraction expansion of an irrational, $$\mu \left(E_0^d\left(\bm{\alpha}, \varphi_k\right) \right) \,=\, \sum_{\alpha_{d+1}=\varphi_k}^{\infty} \mu\left(\widetilde{E_0^d}\left(\bm{\alpha}, \alpha_{d+1}\right)\right).$$ Now, it follows from~(\ref{lebesguegaussmesures}) that, given $\alpha_{d+1}\ge \varphi_k$, $$\frac{\lambda\left(\widetilde{E_0^d}\left(\bm{\alpha}, \alpha_{d+1}\right)\right)}{2\log 2} \, \le \, \mu\left(\widetilde{E_0^d}\left(\bm{\alpha}, \alpha_{d+1}\right)\right) \, \le \, \frac{\lambda\left(\widetilde{E_0^d}\left(\bm{\alpha}, \alpha_{d+1}\right)\right)}{\log 2}\cdotp$$ Furthermore, denoting $\bm{\alpha}\in \llbracket 1, A \rrbracket^d$ by $\bm{\alpha}=(\alpha_1, \dots, \alpha_d)$, (\ref{lemdvlptfractcontinue6}) and~(\ref{lemdvlptfractcontinue7a}) imply that $$\frac{1}{2\,(2A)^{2d}\,\alpha_{d+1}^2}\, \le \, \frac{1}{2^{2d+1}\,\prod_{k=1}^{d+1}\alpha_k^2}\, \le \, \lambda\left(\widetilde{E_0^d}\left(\bm{\alpha}, \alpha_{d+1}\right)\right)\, \le \, \frac{1}{\prod_{k=1}^{d+1}\alpha_k^2}\, \le\, \frac{1}{\alpha_{d+1}^2},$$ hence, on the one hand, $$\mu \left(E_0^d\left(\bm{\alpha}, \varphi_k\right) \right) \, \le \, \frac{1}{\log 2}\sum_{\alpha_{d+1}=\varphi_k}^{\infty}\frac{1}{\alpha_{d+1}^2}\, \le \, \frac{1}{(\log 2)\,\varphi_k}$$ and, on the other, $$\mu \left(E_0^d\left(\bm{\alpha}, \varphi_k\right) \right) \, \ge \,\frac{1}{4\,(2A)^{2d}\,\log 2} \sum_{\alpha_{d+1}=\varphi_k}^{\infty}\frac{1}{\alpha_{d+1}^2} \, \ge \, \frac{1}{4\,(2A)^{2d}\,\log 2\,(\varphi_k+1)}\cdotp$$ Thus, it has been proved that, for any $k\ge 0$ and any $\bm{\alpha}\in \llbracket 1, A \rrbracket^d$, 
\begin{equation}\label{concl1ereetape}
\frac{1}{4\,(2A)^{2d}\,\log 2\,(\varphi_k+1)} \, \le \, \mu \left(E_0^d\left(\bm{\alpha}, \varphi_k\right) \right) \, \le \, \frac{1}{(\log 2)\,\varphi_k}\cdotp
\end{equation}

\paragraph{Step 2.} The second step consists of finding \sloppy a lower bound for $\mu\left(E_{ck}^d\left(\bm{f}_{ck}, \varphi_{ck}\right)\right)$ for $k$ large enough depending on a fixed parameter $\epsilon\in (0,1)$. 

Let $k\ge 1$~: it should be clear that 
\begin{align*}
\mu\left(E_{ck}^d\left(\bm{f}_{ck}, \varphi_{ck}\right)\right) \, &= \, \sum_{\bm{\alpha}\in\llbracket 1, A \rrbracket^d} \mu\left(E_{ck}^d\left(\bm{f}_{ck}, \varphi_{ck}\right)\cap \bm{f}_{ck}^{-1}\left(\left\{\bm{\alpha}\right\}\right)\right) \\ 
&= \, \sum_{\bm{\alpha}\in\llbracket 1, A \rrbracket^d} \mu\left(E_{ck}^d\left(\bm{\alpha}, \varphi_{ck}\right)\cap \bm{f}_{ck}^{-1}\left(\left\{\bm{\alpha}\right\}\right)\right),
\end{align*} 
whence, using the $T$ invariance of the measure $\mu$ and Lemma~\ref{lemme1aborelberni}, 
\begin{align*}
\mu\left(E_{ck}^d\left(\bm{f}_{ck}, \varphi_{ck}\right)\right) \, &= \, \sum_{\bm{\alpha}\in\llbracket 1, A \rrbracket^d} \mu\left(E_{0}^d\left(\bm{\alpha}, \varphi_{ck}\right)\cap T^{-ck}\left(\bm{f}_{ck}^{-1}\left(\left\{\bm{\alpha}\right\}\right)\right)\right) \\ 
&= \, \sum_{\bm{\alpha}\in\llbracket 1, A \rrbracket^d} \mu\left(E_{0}^d\left(\bm{\alpha}, \varphi_{ck}\right)\right)\,\mu\left(\bm{f}_{ck}^{-1}\left(\left\{\bm{\alpha}\right\}\right)\right)\,\left(1+O\left(\theta^{\sqrt{ck-d}} \right)\right).
\end{align*}
Let $\epsilon \in (0,1)$. Choose $k_0\ge 1$ \sloppy large enough so that for all $k\ge k_0$, $\left(1+O\left(\theta^{\sqrt{ck-d}} \right)\right)\ge 1-\epsilon$. It then follows from~(\ref{concl1ereetape}) that 
\begin{align}
\mu\left(E_{ck}^d\left(\bm{f}_{ck}, \varphi_{ck}\right)\right) \, &\ge \, \frac{1-\epsilon}{4\,(2A)^{2d}\,\log 2 \, (\varphi_{ck}+1)}\sum_{\bm{\alpha}\in\llbracket 1, A \rrbracket^d} \mu\left(\bm{f}_{ck}^{-1}\left(\left\{\bm{\alpha}\right\}\right)\right) \nonumber\\ 
&= \, \frac{1-\epsilon}{4\,(2A)^{2d}\,\log 2 \, (\varphi_{ck}+1)}\cdotp \label{concl2emeetape}
\end{align}

\paragraph{Step 3.} The third step consists of finding \sloppy an upper bound for $\mu\left(E_{ck}^d\left(\bm{f}_{ck}, \varphi_{ck}\right)\cap E_{cl}^d\left(\bm{f}_{cl}, \varphi_{cl}\right)\right)$ for $k$ and $l$ large enough depending on a fixed parameter $\epsilon\in (0,1)$. 

Let $l\ge k\ge 1$~: it should be clear that 
\begin{align*}
\mu &\left( E_{ck}^d \left( \bm{f}_{ck}, \varphi_{ck}\right)\cap E_{cl}^d\left( \bm{f}_{cl}, \varphi_{cl}\right)\right) \\ 
&= \, \sum_{\bm{\alpha_1}, \bm{\alpha_2}\in\llbracket 1, A\rrbracket ^d} \mu\left(E_{ck}^d\left(\bm{f}_{ck}, \varphi_{ck}\right)\cap E_{cl}^d\left(\bm{f}_{cl}, \varphi_{cl}\right)\cap \bm{f}_{ck}^{-1}\left(\left\{\bm{\alpha_1}\right\}\right)\cap \bm{f}_{cl}^{-1}\left(\left\{\bm{\alpha_2}\right\}\right)\right)\\
&=\, \sum_{\bm{\alpha_1}, \bm{\alpha_2}\in\llbracket 1, A\rrbracket ^d} \mu\left(E_{ck}^d\left(\bm{\alpha_1}, \varphi_{ck}\right)\cap E_{cl}^d\left(\bm{\alpha_2}, \varphi_{cl}\right)\cap \bm{f}_{ck}^{-1}\left(\left\{\bm{\alpha_1}\right\}\right)\cap \bm{f}_{cl}^{-1}\left(\left\{\bm{\alpha_2}\right\}\right)\right),
\end{align*}
whence, using the $T$ invariance of the measure $\mu$ and Lemma~\ref{lemme1aborelberni}, 
\begin{align*}
\mu &\left( E_{ck}^d \left( \bm{f}_{ck}, \varphi_{ck}\right)\cap E_{cl}^d\left( \bm{f}_{cl}, \varphi_{cl}\right)\right) \\ 
&= \, \sum_{\bm{\alpha_1}, \bm{\alpha_2}\in\llbracket 1, A\rrbracket ^d} \mu\left(E_{0}^d\left(\bm{\alpha_1}, \varphi_{ck}\right)\cap T^{-ck}\left(E_{cl}^d\left(\bm{\alpha_2}, \varphi_{cl}\right)\cap \bm{f}_{ck}^{-1}\left(\left\{\bm{\alpha_1}\right\}\right)\cap \bm{f}_{cl}^{-1}\left(\left\{\bm{\alpha_2}\right\}\right)\right)\right)\\
&= \, \sum_{\bm{\alpha_1}, \bm{\alpha_2}\in\llbracket 1, A\rrbracket ^d} \mu\left(E_{0}^d\left(\bm{\alpha_1}, \varphi_{ck}\right)\right)\, \mu\left(E_{cl}^d\left(\bm{\alpha_2}, \varphi_{cl}\right)\cap \bm{f}_{ck}^{-1}\left(\left\{\bm{\alpha_1}\right\}\right)\cap \bm{f}_{cl}^{-1}\left(\left\{\bm{\alpha_2}\right\}\right)\right)\, \left(1+O\left( \theta^{\sqrt{ck-d}}\right) \right)\\
&= \, \left(1+O\left( \theta^{\sqrt{ck-d}}\right) \right)\times\\
&\qquad \left(\sum_{\bm{\alpha_1}, \bm{\alpha_2}\in\llbracket 1, A\rrbracket ^d} \mu\left(E_{0}^d\left(\bm{\alpha_1}, \varphi_{ck}\right)\right)\, \mu\left(E_{0}^d\left(\bm{\alpha_2}, \varphi_{cl}\right)\cap T^{-cl}\left(\bm{f}_{ck}^{-1}\left(\left\{\bm{\alpha_1}\right\}\right)\cap \bm{f}_{cl}^{-1}\left(\left\{\bm{\alpha_2}\right\}\right)\right)\right)\right)\\
&= \, \left(1+O\left( \theta^{\sqrt{ck-d}}\right)\right)\left(1+O\left( \theta^{\sqrt{cl-d}}\right)\right)\times\\
&\qquad \left( \sum_{\bm{\alpha_1}, \bm{\alpha_2}\in\llbracket 1, A\rrbracket ^d} \mu\left(E_{0}^d\left(\bm{\alpha_1}, \varphi_{ck}\right)\right)\, \mu\left(E_{0}^d\left(\bm{\alpha_2}, \varphi_{cl}\right)\right)\, \mu\left(\bm{f}_{ck}^{-1}\left(\left\{\bm{\alpha_1}\right\}\right)\cap \bm{f}_{cl}^{-1}\left(\left\{\bm{\alpha_2}\right\}\right)\right)\right).
\end{align*}
Given $\epsilon\in (0,1)$, choose $k_0\ge 1$ large enough so that for all $k\ge k_0$, $\left(1+O\left(\theta^{\sqrt{ck-d}} \right)\right)\le 1+\epsilon$. It then follows from~(\ref{concl1ereetape}) that, for all $l\ge k \ge k_0$,
 \begin{align}
\mu \left( E_{ck}^d \left( \bm{f}_{ck}, \varphi_{ck}\right)\cap E_{cl}^d\left( \bm{f}_{cl}, \varphi_{cl}\right)\right) \, & \le \, \frac{(1+\epsilon)^2}{\left(\log 2 \right)^2}\cdotp \frac{1}{\varphi_{ck}\varphi_{cl}} \sum_{\bm{\alpha_1}, \bm{\alpha_2}\in\llbracket 1, A\rrbracket ^d} \mu\left(\bm{f}_{ck}^{-1}\left(\left\{\bm{\alpha_1}\right\}\right)\cap \bm{f}_{cl}^{-1}\left(\left\{\bm{\alpha_2}\right\}\right)\right) \nonumber\\ 
&= \, \frac{(1+\epsilon)^2}{\left(\log 2 \right)^2 \varphi_{ck}\varphi_{cl}} \cdotp \label{concl3emeetape}
\end{align}

\paragraph{Step 4.} Let $\epsilon\in (0,1)$. Choose $k_0\ge 1$ large enough so that the conclusions of steps 2 and 3 hold for all $l\ge k\ge k_0$. By assumption on the integer $c\ge d+1$ in Lemma~\ref{lemborelbernstein} and from~(\ref{concl2emeetape}), the series $\sum_{k\ge 1} \mu\left(E_{ck}^d\left(\bm{f}_{ck}, \varphi_{ck}\right)\right)$ diverges. Therefore, Lemma~\ref{reciproqueborelcantelli} applies and, using~(\ref{concl2emeetape}), (\ref{concl3emeetape}) and noticing that $\limsup_{k\rightarrow \infty} E_{ck}^d\left(\bm{f}_{ck}, \varphi_{ck}\right) \subset \mathcal{S}^d\left(\bm{f}, \varphi\right)$, one gets
\begin{align*}
\mu\left(\mathcal{S}^d\left(\bm{f}, \varphi\right)\right) \, & \ge \, \limsup_{i\rightarrow \infty} \left(\left(\frac{\log 2\,(1-\epsilon)}{4\,(2A)^{2d}\,\log 2 \,(1+\epsilon)}\right)^2 \frac{\left(\sum_{k=k_0}^i (\varphi_{ck}+1)^{-1}\right)^2}{\sum_{k_0\le k , l \le i}(\varphi_{ck}\varphi_{cl})^{-1}}\right)\\
& \ge \, \left(\frac{1-\epsilon}{1+\epsilon}\right)^2\frac{1}{\left(8\,(2A)^{2d}\right)^2}\cdotp
\end{align*}
The result then follows from~(\ref{lebesguegaussmesures}) on letting $\epsilon$ tend to 0.
\end{proof}

\subsubsection{Completion of the proof of Corollary~\ref{resultatprincikhintchiunif}}

The completion of the proof of Corollary~\ref{resultatprincikhintchiunif} requires the introduction of a final two lemmas. The first one is well--known and the second one is elementary.

\begin{lem}\label{loi01applicationergodique}
Let $l\ge 2$ and $q\ge 0$ be integers. Consider the map $$\widetilde{T}~: x \imod{a} \, \mapsto \, lx+\frac{q}{l} \imod{a}.$$ Let $A\subset [0,a)$ such that $\widetilde{T}(A)\subset A$. 

Then $\lambda(A)\in\{0,a \}$.
\end{lem}

\begin{proof}
See for instance Lemma 7 in~\cite{sprin}.
\end{proof}

\begin{lem}\label{lemelemzsurbz}
Let $\alpha, \beta\in\Z/b\Z$. 

Then, there exist $i_1(\alpha, \beta)$ and $i_2(\alpha, \beta)$ in $\Z/b\Z$ such that, defining $$u_{-1}=\alpha, \quad u_0 = \beta, \quad u_1 = i_1(\alpha, \beta)u_0+u_{-1} \quad\mbox{and}\quad u_2 = i_2(\alpha, \beta)u_1+u_0,$$ one has $u_2=0$ in $\Z/b\Z$.
\end{lem}

\begin{proof}
All equations in this proof must be read in $\Z/b\Z$.

If $\alpha=0$ (resp.~$\beta=0$), the choices $i_1(0, \beta)=1$ and $i_2(0, \beta)=-1$ (resp.~$i_1(\alpha, 0)=i_2(\alpha,0)=0$) independently of $\beta\in\Z/b\Z$ (resp.~of $\alpha\in\Z/b\Z$) are easily seen to satisfy the conclusion of the lemma.

Assume therefore that $\alpha\neq 0$ and $\beta\neq 0$. Viewing $\alpha$ and $\beta$ as integers in the interval $\llbracket 0, b-1\rrbracket$, one can then write $u_k = \gcd(\alpha, \beta)v_k$ for $k=-1,\dots,2$, where the finite sequence $\left(v_k \right)_{-1\le k \le 2}$, well--defined in $\Z/\gcd(\alpha,\beta, b)\Z$, satisfies in this ring a recurrence relation similar to that of $\left(u_k\right)_{-1\le k \le 2}$ in $\Z/b\Z$. Even if it means proving the result for the lift to $\Z /b \Z$ of the sequence $\left(v_k \right)_{-1\le k \le 2}$ which satisfies the conditions $v_{-1} \equiv \alpha/\gcd(\alpha, \beta) \imod{b}$ and $v_0 \equiv \beta/\gcd(\alpha, \beta) \imod{b}$, it may be assumed without loss of generality that $\gcd(\alpha, \beta)=1$.

From Dirichlet's theorem on arithmetic progressions, the sequence of integers $\left(\alpha +i\beta\right)_{i\ge 0}$ contains infinitely many primes. Therefore, there exists $i\in\Z/b\Z$ such that $\alpha + i\beta$ is invertible in $\Z/ b\Z$. Setting $i_1(\alpha, \beta)=i$ and $i_2(\alpha, \beta)=-u_0 u_1^{-1}$ leads to the result.
\end{proof}

\begin{proof}[Completion of the proof of Corollary~\ref{resultatprincikhintchiunif}]
It is well--known that, for almost all $\xi\in [0,1]\backslash\Q$, $$\lim_{k\rightarrow \infty} \sqrt[k]{q_k\left( \xi\right)}\, = \, \exp\left(\frac{\pi^2}{12\log 2} \right).$$ This follows for instance from Birkhoff's pointwise ergodic theorem applied to the ergodic system $\left(T,\mu, \mathcal{B}_{[0,1]} \right)$  introduced in the preceding subsection --- see Corollary~3.8 from~\cite{einsiedlerward} for details.

In particular, there exist two positive constants $B$ and $B'$ such that, for almost all $\xi\in [0,1]\backslash\Q$, there exists an integer $k_0$ depending on $B, B'$ and $\xi$ such that, for all $k\ge k_0$, $$\exp\left(B'k\right)\, \le \, q_k\left(\xi\right)\, \le \, \exp\left(Bk\right).$$ 
Set $\tau=\left( 8(ab)^2\max\{4,\widetilde{\Psi}(1)^{-1}\}\right)^{-1}$, which corresponds to the inverse of the constant given in~(\ref{cxiexpression}) (with $M=1$) for natural choices of the parameters $\kappa, \eta$ and $\gamma$ under the assumption of the monotonicity of $\widetilde{\Psi}$. Then, by virtue of Theorem~\ref{CNSapproxunifarithnonmetr} and Remark~\ref{etsikgek_0}, one gets on the one hand that, almost surely, 
\begin{equation}\label{ens1adivgce}
\left\{\xi\in [0,1]\backslash\Q \; :\; \exists k_0\left(B',\xi\right)\ge 0, \,\,\forall k\ge k_0\left(B',\xi\right), \,\, a_k\left(\xi\right)\le \tau\widetilde{\Psi}\left(e^{B'k}\right)  \right\} \, \subset \, \mathcal{U}\left( \Psi\right).
\end{equation}
On the other, from~(\ref{cxiexpression}), it should be clear that, almost surely,
\begin{align}
\bigcap_{M=1}^{\infty} &\left\{\xi\in [0,1]\backslash\Q \; :\; a_k\left(\xi\right)\ge M\widetilde{\Psi}\left(e^{Bk}\right)\; \mbox{ and }\; b|q_{k-1}\left(\xi\right) \; \mbox{i.o.} \right\} \nonumber\\ 
&\qquad \subset \, \bigcap_{c=1}^{\infty}\left([0,1]\backslash\mathcal{U}\left(c\Psi\right)\right)\, = \, [0,1]\backslash\mathcal{V}\left(\Psi\right), \label{ens1bcvgce}
\end{align}
where $\mathcal{V}\left(\Psi\right)$ has been defined in~(\ref{Vpsi}).

Notice also that for any $C>0$, the two series $\sum_{Q\ge 1}\left(\widetilde{\Psi}\left(e^{CQ}\right)\right)^{-1}$ and $\sum_{Q\ge 1}\left(Q^2\Psi(Q)\right)^{-1}$ converge (resp.~diverge) simultaneously. This follows from the change of variable $y=e^{Cx}$ in the corresponding integral $\displaystyle\int\frac{\textrm{d}x}{\widetilde{\Psi}\left(e^{Cx}\right)}$ under the assumption of the monotonicity of $ \widetilde{\Psi}$.

Assume first that the series $\sum_{Q\ge 1}\left(Q^2\Psi(Q)\right)^{-1}$ converges. It then follows from Lemma~\ref{lemborelbernstein} and Remark~\ref{remborelberni} that, for almost all $\xi\in [0,1]\backslash\Q$, there exist only finitely many indices $k\ge 0$ such that $a_k\left(\xi\right)\ge \tau\widetilde{\Psi}\left(e^{B'k}\right)$. Therefore, the set in the left--hand side of~(\ref{ens1adivgce}) has full measure, which completes the proof in this case.

Assume now that the series $\sum_{Q\ge 1}\left(Q^2\Psi(Q)\right)^{-1}$ diverges. From Lemma~\ref{lemelemzsurbz} and formulae~(\ref{lemdvlptfractcontinue5}), for any pair $(\alpha, \beta)\in\llbracket 1,b \rrbracket^2$ and any integer $k\ge 4$ , there exists $(a_{k-2}, a_{k-1}) = (i_1\left(\alpha, \beta\right), i_2\left(\alpha, \beta\right))\in \llbracket 1,b \rrbracket^2$ such that, if $q_{k-4}\equiv \alpha \imod{b}$ and $q_{k-3}\equiv \beta \imod{b}$, then $q_{k-1}\equiv 0 \imod{b}$. Apply then Lemma~\ref{lemborelbernstein} with $A=b$, $d=2$, $c=3$, $\varphi = \left(\widetilde{\Psi}\left(e^{Bk}\right)\right)_{k\ge 0}$ and, for $k\ge 4$, $$\bm{f}_k~: \xi\in [0,1]\backslash\Q \mapsto (i_1\left(\alpha, \beta\right), i_2\left(\alpha, \beta\right)) \in \llbracket 1, b\rrbracket^2 \; \mbox{ if } (q_{k-4}\left(\xi \right), q_{k-3}\left(\xi \right))\equiv (\alpha, \beta) \imod{b}.$$ For such a choice of $\varphi$ and of $\bm{f}:=\left(\bm{f}_k\right)_{k\ge 4}$, consider the sequence of sets $\left(\mathcal{S}^2\left(\bm{f}, M\varphi\right)\right)_{M\ge 1}$ as defined in Lemma~\ref{lemborelbernstein}. It should be clear that this is a sequence decreasing for inclusion and that, for any $M\ge 1$, \begin{equation}\label{S2fMphi}
\mathcal{S}^2\left(\bm{f}, M\varphi\right) \, \subset \, \left\{\xi\in [0,1]\backslash\Q \; :\; a_k\left(\xi\right)\ge M\widetilde{\Psi}\left(e^{Bk}\right)\; \mbox{ and }\; b|q_{k-1}\left(\xi\right) \; \mbox{i.o.} \right\}.
\end{equation} 
Furthermore, from the Monotone Convergence Theorem and Lemma~\ref{lemborelbernstein}, $$\lambda\left(\bigcap_{M=1}^{\infty}\mathcal{S}^2\left(\bm{f}, M\varphi\right)\right)\, = \, \lim_{M\rightarrow \infty}\lambda\left(\bigcap_{M'=1}^{M}\mathcal{S}^2\left(\bm{f}, M'\varphi\right)\right)\, = \, \lim_{M\rightarrow \infty}\lambda\left(\mathcal{S}^2\left(\bm{f}, M\varphi\right)\right)\, \ge \, \frac{\log 2}{2^{14}b^8}\cdotp$$
Combining this last inequality with~(\ref{ens1bcvgce}) and~(\ref{S2fMphi}) shows that the complement of the set $\mathcal{V}\left(\Psi\right)$ has strictly positive measure. Now, it should be clear from its definition in~(\ref{Vpsi}) that the set $\mathcal{V}\left(\Psi\right)$ is invariant under the map $ x \imod{a} \, \mapsto \, tx \imod{a}$, where $t\ge a$ is any integer congruent to 1 modulo $a$. From Lemma~\ref{loi01applicationergodique}, this implies that the complement of $\mathcal{V}\left(\Psi\right)$ in $[0,a]$ has full measure, that is, that $$\lambda\left(\mathcal{U}\left(\Psi\right)\right)\, \le \, \lambda\left(\mathcal{V}\left(\Psi\right)\right) \, =\, 0.$$ This completes the proof of Corollary~\ref{resultatprincikhintchiunif}.
\end{proof}

\section{Some applications}\label{applicationapproxarithm}

Some of the applications of the theory developed in this paper are mentioned in this section.

A Dirichlet type result can always be used to obtain bounds for certain types of exponential sums. In this respect, Theorem~\ref{CNSapproxunifarithnonmetr} may help to improve or specify some exponential sums when the numerators and the denominators of the rational approximants are restricted to prescribed arithmetic progressions --- see for example~\cite{erdosrenyi} or p.172 of~\cite{tenenbaumintro}. On the other hand, Walfisz proved in~\cite{walfiszelliptik} a very particular case of the Khinctchine type result given by Theorem~\ref{alakhint} in order to study the behaviour of the elliptic function $$\vartheta(z) = \sum_{n=-\infty}^{\infty}z^{n^2}$$ near its circle of convergence. 

Two specific applications of Theorem~\ref{approxasympnonmetr} and Corollary~\ref{resultatprincikhintchiunif} shall now be developed. The first one is mainly due to S.Hartman who was the first to notice in~\cite{hartmanconditionsupp} that a result such as Theorem~\ref{approxasympnonmetr} enables one to determine the value of $\liminf_{n\rightarrow \infty} \left(\sin n\right)^n$. This can be generalized thanks to the inhomogeneous version of Theorem~\ref{approxasympnonmetr} mentioned in Remark~\ref{generalizthapproxasymnonmetr}.

\begin{prop}
Let $\xi$ be an irrational which is not a rational multiple of $\pi$. Let also $\alpha\in\R$.

Then, $$\liminf_{n\rightarrow \infty} \left(\sin (n\xi +\alpha)\right)^n = \liminf_{n\rightarrow \infty} \left(\cos (n\xi +\alpha)\right)^n = -1$$ and  $$\limsup_{n\rightarrow \infty} \left(\sin (n\xi +\alpha)\right)^n =\limsup_{n\rightarrow \infty} \left(\cos (n\xi +\alpha)\right)^n =1.$$
\end{prop}

\begin{proof}
It shall be proved that $\liminf_{n\rightarrow \infty} \left(\sin (n\xi +\alpha)\right)^n =-1$. All the other equations can be established in a similar fashion.

From Remark~\ref{generalizthapproxasymnonmetr}, there exist two sequences of integers $\left(c_n\right)_{n\ge 1}$ and $\left(d_n\right)_{n\ge 1}$ with $d_n\ge 1$ and $\lim_{n\rightarrow \infty}d_n = \infty$ such that, for all $n\ge 1$, 
\begin{align}
&\left|d_n\frac{\pi}{2\xi}-c_n -\frac{\alpha}{\xi}\right|\, \le \, \frac{2}{d_n},\label{applitrigo1}\\
& d_n \equiv 3\imod{4}\;\mbox{ and}\label{applitrigo2}\\
& c_n\equiv 1 \imod{2}.\label{applitrigo3}
\end{align}
Therefore, for all $n\ge 1$,  $$\left|d_n\frac{\pi}{2\xi}-c_n -\frac{\alpha}{\xi}\right|\, = \, O\left(\frac{1}{d_n}\right)\, \underset{(\ref{applitrigo1})}{=} \, O\left(\frac{1}{c_n}\right)\cdotp$$ With the help of a Taylor expansion, this implies that $$\sin (c_n\xi +\alpha)\, = \, \sin \left(d_n\frac{\pi}{2} +O\left(\frac{1}{c_n}\right)\right)\, \underset{(\ref{applitrigo2})}{=} \, -1+O\left(\frac{1}{c_n^2}\right),$$ hence $$\left(\sin (c_n\xi +\alpha)\right)^{c_n}\, = \,\left(-1+o\left(\frac{1}{c_n}\right)\right)^{c_n}\, \overset{(\ref{applitrigo3})}{\underset{n\rightarrow \infty}{\longrightarrow}} \, -1.$$
\end{proof}

The second application is of a geometrical nature and exploits the link between approximation by rationals with numerators and denominators in given arithmetic progressions and pseudo--lattices in dimension 2. More precisely, a natural analogue of P\'olya's orchard problem is now discussed. The latter is formulated in~\cite{polyaszegopbs} (Chap.5, Problem 239) in this form~: ``How thick must be the trunks of the trees in a regularly spaced circular forest grow if they are to block completely the view from the center?''

Assume that the forest (or the orchard) is situated in a disk of integer radius $N\ge ab$ and that each point of the lattice $b\Z\times a\Z$ different from the origin and lying in this disk is the center of a tree of radius $r>0$ (here, $a,b\ge 1$). Minkowski's Convex Body Theorem can then be used to solve the visibility problem above and to obtain that the choice of $r=ab/N$ blocks the view from the center (see for instance Lemma~3 in~\cite{kruskyforet}). Allen in~\cite{allenforet} computed the infimum of all radii of trees preventing an observer situated at the origin from seeing a point outside the forest and Kruskal generalized this result to more general configurations of trees (see~\cite{kruskyforet}).

In what follows, the horizon will be said to be visible from the origin in the direction given by a line $\Delta$ passing through the origin if, given a forest of a prescribed type lying in the half--plane $\{x\ge 0\}$, the line $\Delta$ does not intersect any of the trees in the forest. 

For the forests under consideration, the latter will be planted in a subset of $\Lambda\cap \{x\ge 0\}$, where $\Lambda$ is the pseudo--lattice $\Lambda = (b\Z +s)\times (a\Z+r)$, with $a,b,r$ and $s$ satisfying~(\ref{contraintesdebase}) and $r\neq 0$ or $s\neq 0$. 

The connection between Diophantine approximation and the problem of visibility is then given by this simple fact~: an inequality of the type $|(bn+s)\xi-(am+r)|\le c$, where $c\ge 0$ is real, $\xi$ is irrational and $(m,n)\in\Z^2$, precisely means that the vertical segment joining the point $(bn+s, am+r)\in\Lambda$ to the line $\Delta~: y=\xi x$ has a length less than $c$. Therefore, the intersection between $\Delta$ and the closed ball centered at $(bn+s, am+r)$ with radius $c$ is non--empty~: if the latter ball represents a tree in a forest, the horizon is not visible in the direction given by $\Delta$.

For the sake of simplicity, the results will be stated from a qualitative point of view~: although possible, none of the constants mentioned below will be made effective.

\paragraph{Geometrical interpretation of Theorem~\ref{approxasympnonmetr}.} The forest is defined this way~: a tree of radius $(ab)/(4(bn+s))$ is planted at each point $(bn+s, am+r)\in\Lambda\cap\{x> 0\}$. The observer is situated at the origin in a glade of any shape but with bounded diameter (cf.~Figure~\ref{foretasyfigure}).

From Theorem~\ref{approxasympnonmetr}, for any line of sight with irrational slope, the observer will never see the horizon, no matter how big the glade is. From Remark~\ref{rationelsapproxasympnonmetr}, it is however possible to see the horizon along a direction given by a line with rational slope if, for instance, the glade contains a disk centered at the origin with sufficiently large radius. On the other hand, the validity of Conjecture~\ref{conjabsur4} would imply that there exist angles of sight with irrational slope if the constant $ab/4$ were to be replaced by another one small enough in the value of the radii of the trees (again, provided that the glade at the origin is \emph{big enough}).

\begin{minipage}{0.4\textwidth}
\begin{figure}[H]
\includegraphics[height=8cm]{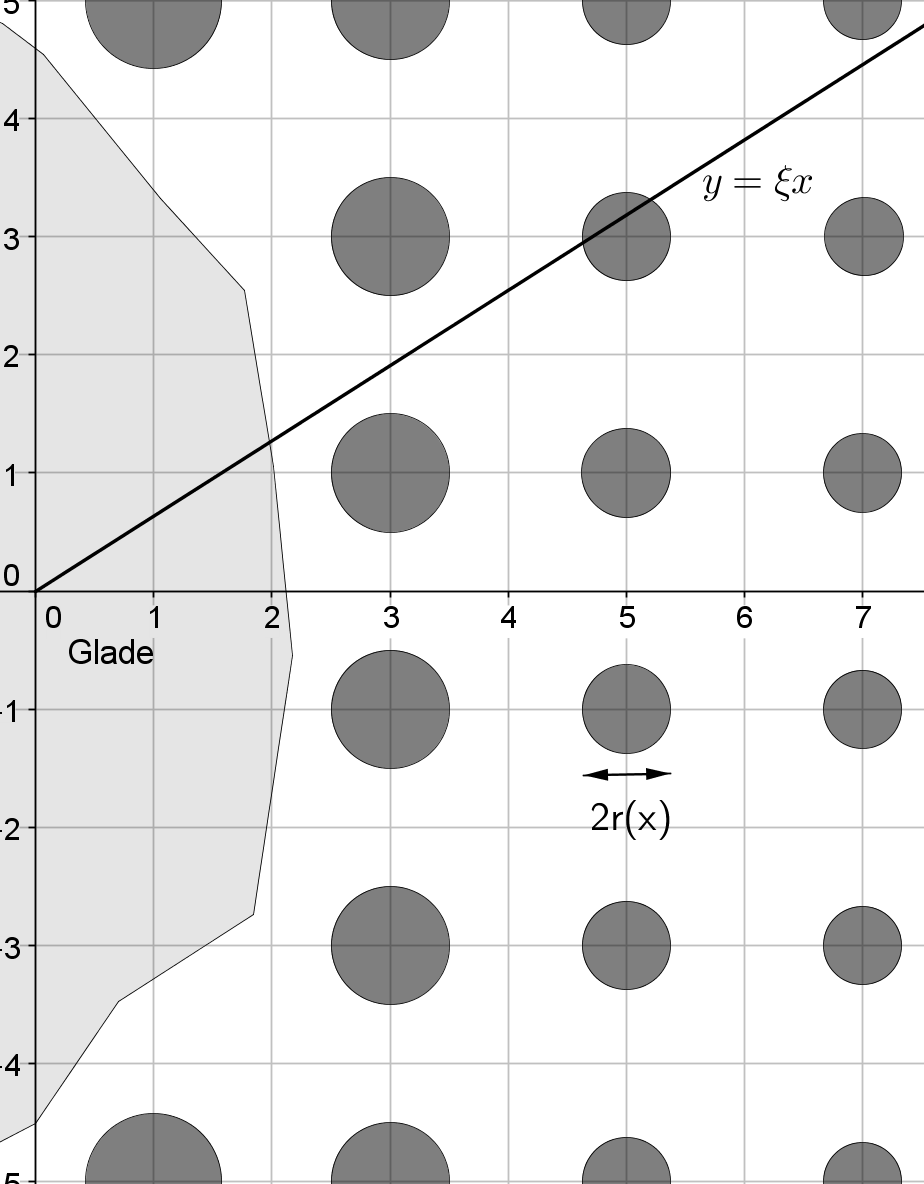}
\caption{{\footnotesize Geometric interpretation of Theorem~\ref{approxasympnonmetr} in the case of the pseudo--lattice $\Lambda = (2\Z +1)^2$. Each tree centered at $(2n+1, 2m+1)$ has radius $r(2n+1)=1/|2n+1|$.}}
\label{foretasyfigure}
\end{figure}
\end{minipage}
\hspace{8ex}
\begin{minipage}{0.4\textwidth}
\begin{figure}[H]
\includegraphics[height=8cm]{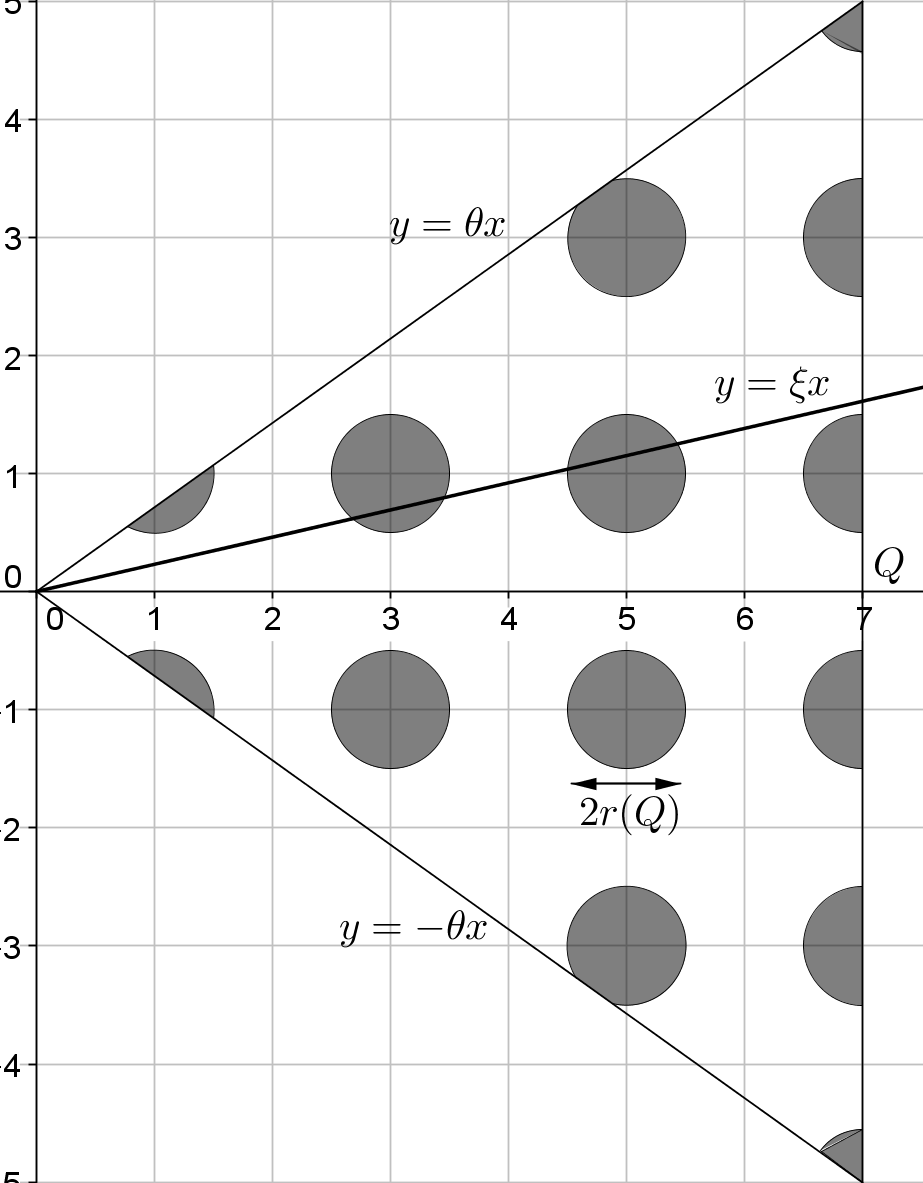}
\caption{{\footnotesize Geometric interpretation of Corollary~\ref{resultatprincikhintchiunif} in the case of the pseudo--lattice $\Lambda = (2\Z +1)^2$. The orchard has depth $Q$ and all the trees have the same radius $r(Q)>0$.}}
\label{foretuniffigure}
\end{figure}
\end{minipage}

\paragraph{Geometrical interpretation of Corollary~\ref{resultatprincikhintchiunif}.} Given $\theta >0$ and $Q\ge 1$, the forest --- which will more conveniently be referred to as an orchard --- is defined this way (see also Figure~\ref{foretuniffigure})~: a tree of radius $r(Q)>0$ is planted at each element of the set $\Lambda \cap \mathcal{L}$, where $$\mathcal{L}:=\left\{(x,y)\in\R^2\; : \; 0\le x\le Q \;\textrm{ and }\; -\theta x \le y \le \theta x \right\}.$$ From Corollary~\ref{resultatprincikhintchiunif}, if the radius of the trees is chosen in such a way that $r(Q)=\log Q/ Q$, then, for almost all $\xi\in (-\theta, \theta)$, there exist arbitrarily large values of $Q$ such that the horizon is visible in the direction $y=\xi x$. On the other hand, if $r(Q)=(\log Q)^{1+\epsilon}/Q$ for some $\epsilon >0$, then, provided that the depth $Q$ of the orchard is large enough (depending on $\xi$), the horizon is never visible in the direction $y=\xi x$ for almost all $\xi \in (-\theta, \theta)$.

\renewcommand{\abstractname}{Acknowledgements}
\begin{abstract}
The author would like to thank Pat McCarthy for suggesting the problem. He is also indebted to his PhD supervisor Detta Dickinson for discussions which helped to develop ideas put forward. His work is supported by the Science Foundation Ireland grant RFP11/MTH3084.
The author would like to thank the referee for his careful reading of a first draft of the manuscript and for valuable comments.
\end{abstract}

\bibliographystyle{plain}
\bibliography{diophapproxarithm}

\end{document}